\documentclass{article}
\usepackage{graphicx,color,overpic}
\usepackage{amsmath,amssymb,amsthm,a4wide}
\usepackage[all]{xy}
\newcommand*\circled[1]{\tikz[baseline=(char.base)]{
            \node[shape=circle,draw,inner sep=2pt] (char) {#1};}}

\newtheorem{thm}{Theorem}[section]

\newtheorem{lem}[thm]{Lemma}
\newtheorem{prop}[thm]{Proposition}
{\theoremstyle{definition}

\newtheorem{rem}[thm]{Remark}
}

\usepackage{url,graphicx}
\usepackage{caption}
\usepackage{subcaption}
\usepackage{todonotes}
\usepackage{tikz,pgf}
\usetikzlibrary{calc}
\tikzstyle{vertex}=[circle, draw, fill=black, inner sep=0pt, minimum size=4pt]
\tikzstyle{lnode}=[circle,white,draw=black!400!white,fill=black!400!white,inner sep=1pt, font=\scriptsize]

\tikzstyle{edge}=[line width=1.5pt,black!50!white]
\tikzstyle{redge}=[edge,red]
\tikzstyle{bedge}=[edge,blue]

\renewcommand{\P}{\mathbb P}
\newcommand{\C}{\mathbb C}
\newcommand{\R}{\mathbb R}
\renewcommand{\S}{\mathbb S}
\renewcommand{\H}{\mathbb H}
\newcommand{\D}{\mathbb D}

\newcommand{\qi}{\mathbf i}
\newcommand{\qj}{\mathbf j}
\newcommand{\qk}{\mathbf k}
\newcommand{\ci}{\mathrm i}
\newcommand{\eps}{\epsilon}
\newcommand{\SE}{\mathrm{SE}}
\newcommand{\SO}{\mathrm{SO}}

\title{And Yet It Moves:\\
	Paradoxically Moving Linkages in Kinematics}

\author{Josef Schicho,
	JKU Linz, Austria}

\begin{document}

\maketitle

Look at Figure~\ref{fig:ell}: you see a mechanism that is able to draw an ellipse. If you press gently on the green bar 
(connected to the right endpoint of the grey segment which is fixed), then the whole vehicle will start to move and bounce
so that the red point traces the ellipse. Historically, it was a famous challenge in the 19th-century
to find a mechanism that draws a straight line segment. Mathematicians even tried to prove the non-existence
of an exact solution. But then the French engineer Peaucellier and the Russian mathematician Lipkin independently found an
exact solution. % known as the ``Peaucellier/Lipkin inversor''. 
Starting from the mechanism in Figure~\ref{fig:ell}, we can do the same thing as well (even though this was not
the solution of Peaucellier/Lipkin): you can change some of the lengths so that the ellipse degenerates into a line segment
traced twice in a full round.

\begin{figure}[h]
\begin{center}
\includegraphics[width=7cm]{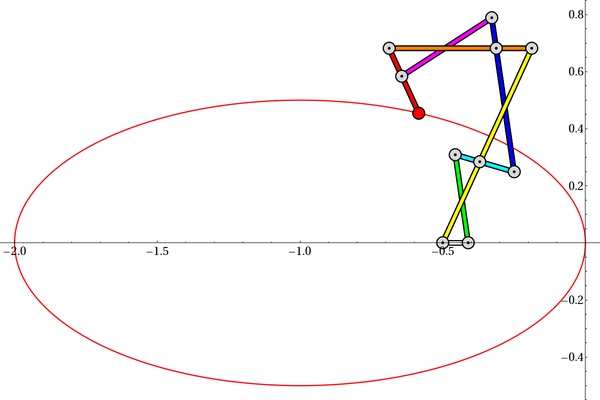}
\caption{A mechanism which is able to draw an ellipse. The short gray horizontal bar is fixed on the x-axis,
whereas all the other bars are allowed to move, according to the rotational joints which link them one to another.
}
\label{fig:ell}
\end{center}
\end{figure}

\paragraph{Kempe's Universality Theorem.}
A few years after the invention of the ``straight line mechanism'' by Peaucellier and Lipkin, 
Kempe~\cite{Kempe} proved that every plane algebraic curve can be drawn by a mechanism moving with one degree of freedom!
His construction uses the implicit equation of the algebraic curve, and the linkage draws a bounded subset of
the curve. Kempe himself admits that the mechanisms constructed by his general construction are quite complicated.
%To be precise, every compact subset of a plane algebraic subset can be drawn by a mechanism with rotational joints.
%Kempe himself did not care so much about such subleties; however, the statement was later made precise and proved 
%by \cite{KapovichMillson:02}. 
One of the objectives in this article is to explain how to construct a mechanism
that draws a given rational curve, i.e., a curve that it is given by a parameterization by rational functions.
Compared with Kempe, this construction gives simpler results when it applies (not every algebraic curve is rational).

\paragraph{Unexpected Mobility.}
Most of the mechanisms in this paper will be {\em paradoxical}, in the following sense: by a systematic
counting of degrees of freedom and constraints, one can estimate if a given mechanism moves. For a paradoxical
mechanism, this estimate predicts that the mechanism is rigid: there are sufficiently many constraints so that
there should be no freedom left for motion, except moving the mechanism as a whole like a rigid body.
Still, the mechanism does move non-trivially. 
We discuss five mathematical tools that somehow ``explain'' the unexpected mobility:
\begin{itemize}
\item edge colorings of graphs;
\item factorization of polynomials over skew coefficient rings;
\item symmetry as a rule changer for counting variables and constraints;
\item a projective duality relating a set of relative positions to a set of geometric parameters;
\item compactification, i.e., a closer analysis of ``limit configurations at infinity''.
\end{itemize}

\paragraph{Links and Joints.}
We need to introduce a few concepts from kinematics 
(please do not worry, we will keep the amount of definitions at a minimal level).
A {\em linkage} (or mechanism) in 3-space is composed of rigid bodies called links (or bars, rods) that are connected
by joints (e.g., hinges or spherical joints); examples occur in mechanical engineering
and robotics, but also in sports medicine -- the human skeleton may be considered as a quite
complex linkage -- and in chemistry, at a microscopic scale. If two links are connected by
a joint, then the type of joint determines a set of possible relative positions of one link with respect to the other.
A {\em revolute joint} (or R-joint or hinge) allows a one-dimensional set of rotations around an axis
which is fixed in both links; this set is a copy of $\SO_2$. This type of joint appears most frequently, for example
in doors and windows or in connection with wheels (see also Figure~\ref{fig:joints}, left). 
A {\em spherical joint} (or S-joint) allows a three-dimensional set of rotations
around a point which is fixed in both links; this set of motions is a copy of $\SO_3$. An example is the
hip joint of the human skeleton (see Figure~\ref{fig:joints}, middle). 
And a {\em prismatic joint} (or P-joint) allows a one-dimensional
set of translations in a fixed direction; this set is theoretically a copy of $\R$, but in reality, it is
a bounded interval. Teachers and students in mathematics often operate such a joint when moving a blackboard up and down 
(see Figure~\ref{fig:joints}, right, for a different example).

\begin{figure}[h]
\begin{center}
\includegraphics[height=3cm]{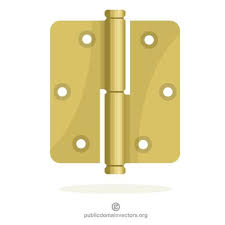}
\hspace{1cm}
\includegraphics[height=3cm]{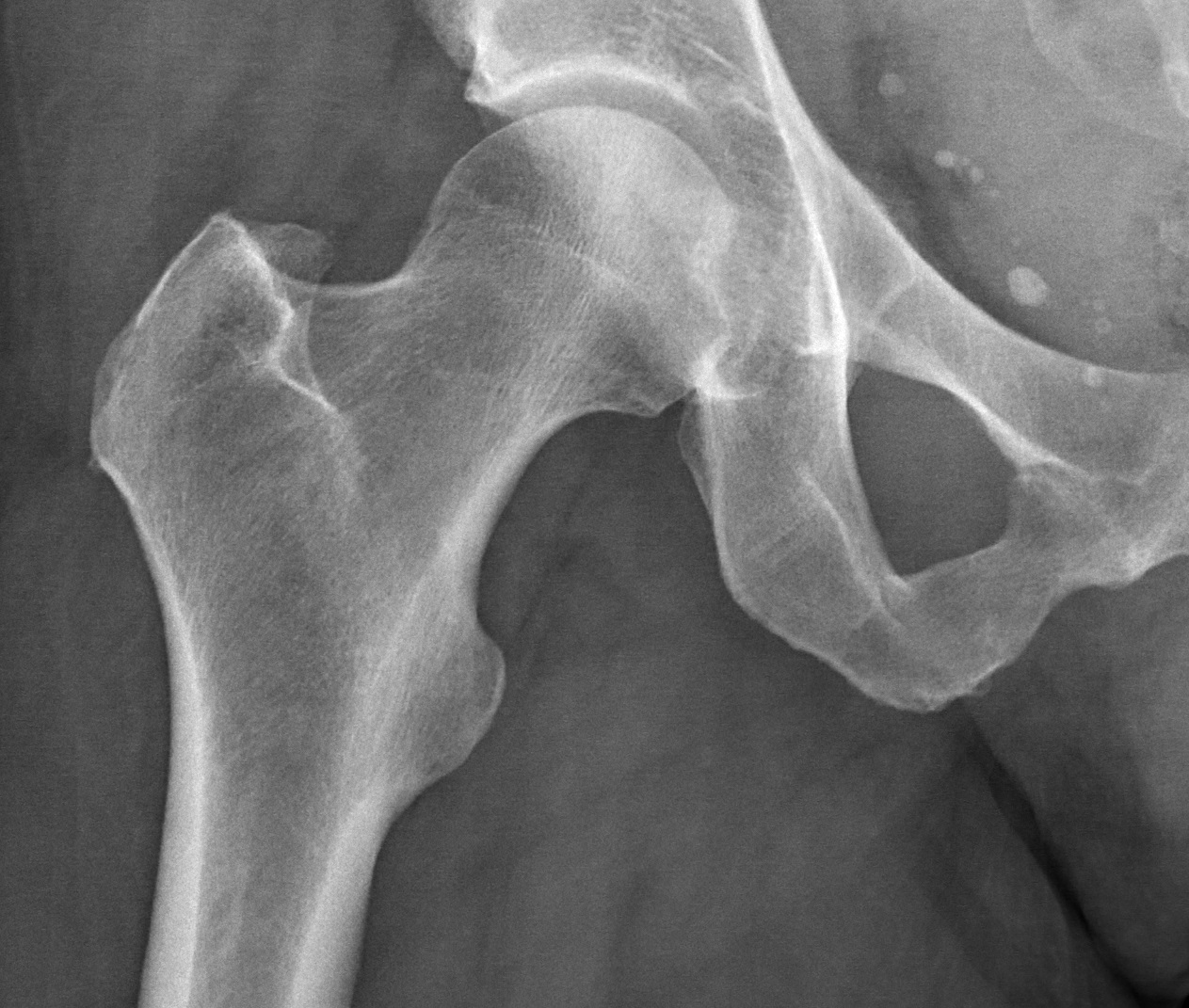}
\hspace{1cm}
\includegraphics[height=3cm]{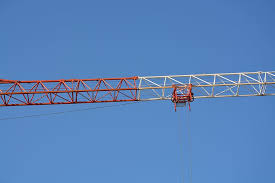}
\caption{
A hinge, the hip joint (spherical), and a prismatic joint on a crane.
}
\label{fig:joints}
\end{center}
\end{figure}

\paragraph{Configurations.}
If two links are not directly connected by a joint, then the set of possible relative positions of one with respect to the other is determined
by other links and joints forming chains that connect the two given links. In general, the description is more
complicated, and it is one of the main tasks of kinematics to determine these sets. In any case, they are
subsets of the group $\SE_3$ of direct isometries, also known as Euclidean displacements. 
The set of all possible relative positions of any pair of rigid bodies of a linkage $L$ is called the {\em configuration space}
of $L$. It is possible to express the constraints coming from the joints by algebraic equations in the joint parameters.
Therefore, the configuration space is an algebraic variety. Its dimension is called the {\em mobility} of $L$. 

A linkage is given by combinatoric data, namely the graph indicating which rigid bodies are connected
by joints and the type of joints such as revolute, spherical, prismatic; and by geometric parameters
determining the fixed position of the joint axis in each of the two links attached to any R-joint and
the fixed position of the anchor point in each of the two links attached to any S-joint.
The computation of the configuration space of a given linkage can be reduced to solving a system of
algebraic equations with parameters, with the size of the system determined by the combinatorics.
These systems form a rich source of computational problems in computer algebra and polynomial
system solving (see \cite{SommeseWampler:11} and the references cited there). 

\paragraph{Structure of the Paper.}
The paper has 6 sections. In Section~\ref{sec:gk}, we discuss combinatoric methods for estimating the
dimension of the configuration space, based on counting variables and equational conditions; this is
necessary to make precise what ``paradoxical'' means. Section~\ref{sec:over} deals with planar linkages
whose links are line segments joined by revolute joints, also known as moving graphs; we discuss 
graphs that should be rigid but actually move. Section~\ref{sec:rl} deals with spatial linkages in the plane
with revolute joints, and uses dual quaternions to construct examples of simply closed linkages that
are paradoxically movable. Section~\ref{sec:sym} deals with symmetries and explains how they can change 
the counting rules. Section~\ref{sec:pod} deals with a particular type of linkage called multipods or Stewart
platforms; here, projective duality is a powerful mathematical tool that allows us to construct paradoxical
examples. Section~\ref{sec:bond} is concerned with the problem of finding necessary conditions for mobility,
based on the idea to analyze the ``configurations at infinity'' of a mobile linkage. In the three
subsections of Section~\ref{sec:bond}, moving graphs, simply closed loops with revolute joints, and
multipods are revisited from this point of view what happens at infinity.

\paragraph{Acknowledgements.}
Matteo Gallet, Georg Grasegger, Christoph Koutschan, Jan Legersky, Zijia Li, Georg Nawratil, and Hans-Peter Schr\"ocker 
are coauthors of papers of which I took pictures - thanks for allowing me to use their work. I also
would like to thank Matteo Gallet, Zijia Li, and Jiayue Qi for helping to improve the narration.
This work has been supported by the Austrian Science Fund (FWF): P31061.

\section{Predicting Mobility} \label{sec:gk}

Given the combinatorics of a linkage, i.e., the number of its rigid bodies and the information
which of them are connected by joints, it is possible to estimate the mobility by counting free
variables and equational constraints. In kinematics, this is called the {\em Chebychev/Gr\"ubler/Kutzbach (CGK) formula}.

\paragraph{Moving Graphs.}
In this section, we start with the two-dimensional situation. Every link is a line segment in the plane $\R^2$.
In the plane, it does not make sense to distinguish revolute joints and spherical joints, and we do not consider prismatic joints.
All joints in the linkages we consider allowing rotations around a fixed point.
The combinatorics of the linkage is conveniently described
by a graph $G=(V,E)$, with vertices corresponding to joints and edges corresponding to links. If a line segment
has three or more (say $k$) joints connecting to other links, then we have to ``split it up'' into several edges: we get 
$k$ vertices corresponding to joints and we connect them by ${k\choose 2}$ edges. 
For instance, the green link in Figure~\ref{fig:ell} will correspond to a triangle in the graph, 
which is geometrically degenerate because its three vertices are collinear.
We assume that the linkage has no ``dangling links'', i.e., no vertices of degree~1, because
they would obviously rotate around the connected vertex.

For a graph $G=(V,E)$, an ``edge length assignment'' is a vector $\lambda\in\R^E$ indexed by the edges with
positive real coordinates $\lambda_e$, $e\in E$. A configuration of $(V,E,\lambda)$ is a collection $(\rho_v)_{v\in V}$ 
with $\rho_v\in\R^2$, such that for any edge $e=(u,v)$, we have $||\rho_u-\rho_v||=\lambda_e$. Two configurations $\rho,\rho'$
are equivalent if there is a direct isometry $\sigma:\R^2\to\R^2$ of the plane such that $\sigma(\rho_v)=\rho'_v$
for all $v$. If we choose two vertices $v,w\in V$ such that $\rho_v\ne\rho_w$, then there is a unique representative $\rho'$
in the equivalence class of $\rho$ such $\rho_v=(0,0)$ and $\rho_w=(0,c)$ for some $c>0$; we then say that
$\rho'$ is a {\em normalized} configuration.

For a given graph $G=(V,E)$ with edge length assignment $\lambda$, its normalized configurations are the solutions
of a system of algebraic equations of the form
\[ (x_a-x_b)^2+(y_a-y_b)^2 = \lambda_{ab}^2 \]
for each edge $\{a,b\}\in E$, and the normalization conditions
\[ x_v=y_v=x_w=0, y_w>0 . \]
The number of nonzero variables is $2|V|-3$, and the number of equations is $|E|$. We leave out the inequality,
because it is inessential for the dimension count. Now the CGK formula predicts that the linkage is
rigid if $2|V|-3=|E|$. If this number is nonnegative, then we call $2|V|-3-|E|$ the CGK estimate for 
the dimension of equivalence classes of configurations. In kinematics, this dimension
is called the {\em mobility} of the linkage.

\begin{figure}[h]
\begin{center} 
    \begin{tikzpicture}[scale=2]
      \node[vertex] (a) at (0,0) {};
      \node[vertex] (b) at (1,0) {};
      \node[vertex] (c) at (0.5,0.5) {};
      \node[vertex] (d) at (0,1.5) {};
      \node[vertex] (e) at (1,1.5) {};
      \node[vertex] (f) at (0.5,1) {};

      \draw[edge] (a)edge(b) (b)edge(c) (c)edge(a) (a)edge(d) (d)edge(e) (e)edge(f) (f)edge(d) (b)edge(e) (c)edge(f);
    \end{tikzpicture}\hspace{2cm}
        \begin{tikzpicture}[scale=2]
      \node[vertex] (a) at (0,0) {};
      \node[vertex] (b) at (1,0) {};
      \node[vertex] (c) at (0.5,0.5) {};
      \node[vertex] (d) at (0,1) {};
      \node[vertex] (e) at (1,1) {};
      \node[vertex] (f) at (0.5,1.5) {};

      \draw[edge] (a)edge(b) (b)edge(c) (c)edge(a) (a)edge(d) (d)edge(e) (e)edge(f) (f)edge(d) (b)edge(e) (c)edge(f);
      \begin{scope}[]
                                \node[vertex] (d2) at (0.6,0.8) {};
                                \node[vertex] (e2) at (1.6,0.8) {};
                                \node[vertex] (f2) at (1.1,1.3) {};

                                \draw[edge,black!20!white] (a)edge(b) (b)edge(c) (c)edge(a) (a)edge(d2) (d2)edge(e2) (e2)edge(f2) (f2)edge(d2) (b)edge(e2) (c)edge(f2);
      \end{scope}
    \end{tikzpicture}
\caption{Two planar linkages with 6 joint and 9 links with the same underlying graph. The left one is rigid,
the right one is mobile.}
\label{fig:3prism}
\end{center}
\end{figure}
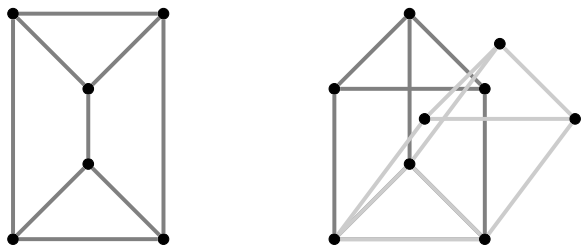

\paragraph{Generic Mobility.}
For a concrete instance, the CGK estimate comes without any warranties. But we can say something definite 
for the ``generic case''. 
Here we use the word ``generic'' in the following sense. 
Assume that a certain statement depends on instances parametrized
by an open subset of an irreducible algebraic variety $P$ (in most cases, $P$ is an open subset of a vector space). 
Then we say that the statement is 
generically true if the subset of instances such that the statement is false is contained in an algebraic subvariety
of $P$ of strictly smaller dimension. 

\begin{prop} \label{prop:cgk}
Let $G=(V,E)$ be a graph. Let $\lambda\in\R^E$ be a generic length assignment. Let $X_\lambda\in\R^{2|V|-3}$ be set 
of normalized configurations of $(V,E,\lambda)$.
If $2|V|-3-|E|\ge 0$, then $X_\lambda$ is either empty or a real manifold of dimension $2|V|-3-|E|$.
In particular, if $2|V|-3-|E|= 0$, then a generic length assignment allows only finitely many normalized configurations.
\end{prop}

\begin{proof}
Let $f:\R^{2|V|-3}\to\R^{|E|}$ be the map $(x_a,y_a)_{a\in V}\mapsto ((x_a-x_b)^2+(y_a-y_b)^2)_{\{a,b\}\in E}$ (in the
domain, remove the three coordinates known to be zero). This is a differential map, which assigns to
each normalized configuration of points in $\R^2$ the square of the lengths of edges. Therefore $X_\lambda = f^{-1}(\lambda)$.

If the image of $f$ does not contain an open neighborhood of $\lambda$, then it also does not contain $\lambda$ because
$\lambda$ is chosen generically. Hence $X_\lambda$ is empty and there is nothing left to prove.

Otherwise, let $U$ be an open neighborhood of $\lambda$ and apply Sard's theorem to the map $f|_{f^{-1}(U)}$.
It implies that the set of critical values does not contain $\lambda$. Hence the Jacobian of $f$ has rank $E$ at every
point of $f^{-1}(U)$, and this shows the claim.
\end{proof}

\paragraph{Generic Rigidity.}
If $|E|=2|V|-3$, then two cases are possible: either the image of the map $f:\R^{2|V|-3}\to\R^{|E|}$ in the proof
contains an open subset. Then the graph is rigid: a generic configuration cannot move continuously, by Proposition~\ref{prop:cgk}.
Or the image of the map is contained in a subset of lower dimension.
The following theorem determines which of the two cases holds.

\begin{thm} \label{thm:laman}
Let $G=(V,E)$ be a graph such that $|E|=2|V|-3$. Then there is an open set of edge assignments $\lambda$ with
a finite and positive number of configurations if and only if $|E'|\le 2|V'|-3$ for every subgraph $G'=(V',E')$ of $G$. 
\end{thm}

This theorem was proven by Pollaczek-Geiringer~\cite{Geiringer:27} and rediscovered 40 years later 
by Laman ~\cite{Laman:70}. The graph that satisfy the necessary and sufficient condition above are called 
{\em Laman graphs}. The necessity is easy to see: if there is a subgraph $G'=(V',E')$ with $|E'|> 2|V'|-3$,
then the algebraic system describing normalized configurations of the subgraph is overdetermined. So, for generic
edge length assignments, there is no configuration for the subgraph, and therefore also no configuration for
the graph $G$ itself.

In dimension~3, the CGK estimate for the mobility of a graph $G=(V,E)$ is equal to $3|V|-6-|E|$. Proposition~\ref{prop:cgk}
holds with that bound: if $\lambda\in\R^{|E|}$ is a generic edge assignment, and the normalized
configuration space $X_\lambda$ is not empty, then it has dimension $3|V|-6-|E|$. The condition $|E'|\le 3|V'|-6$
for every subgraph $(V',E')$ is still necessary for the statement that $X_\lambda$ is generically not empty,
but it is not sufficient: Figure~\ref{fig:l3d} shows the ``double banana'', a graph with 8 vertices and 18 edges, such that a generic
assignment of its vertices to points in $\R^3$ is flexible. The Jacobi matrix of the map $f$ mapping normalized configurations
to edge assignments (see Proposition~\ref{prop:cgk}) is quadratic and singular. So the 3-dimensional analogue
of Theorem~\ref{thm:laman} is not true, and the search for another combinatoric analogue is an active research
topic in rigidity theory (see \cite{Meera}).

\begin{figure}[h]
\begin{center} 
    \begin{tikzpicture}[scale=4]
      \node[vertex] (a) at (0,0) {1}; 
      \node[vertex] (b) at (0,1.5) {2};
      \node[vertex] (c) at (-0.6,0.5) {3};
      \node[vertex] (d) at (-0.5,0.6) {4};
      \node[vertex] (e) at (-0.36,0.55) {5};
      \node[vertex] (f) at (0.6,0.6) {6};
      \node[vertex] (g) at (0.5,0.4) {7};
      \node[vertex] (h) at (0.36,0.55) {8};
  
      \draw[edge]  (a)edge(c) (a)edge(d) (a)edge(e) (b)edge(c) (b)edge(d) (b)edge(e) (c)edge(d) (c)edge(e) (d)edge(e);
      \draw[bedge] (a)edge(f) (a)edge(g) (a)edge(h) (b)edge(f) (b)edge(g) (b)edge(h) (f)edge(g) (f)edge(h) (g)edge(h);
    \end{tikzpicture}
\end{center}
\caption{The smallest graph that is generically mobile and still fulfills the 3D-analogue of Laman's condition
  for generic rigidity: $3|V|-6=|E|$, and $3|V'|-6\ge |E'|$ for every subgraph $(V',E')$. The blue part may revolve
  around the line through two vertices.}
\label{fig:l3d}
\end{figure}
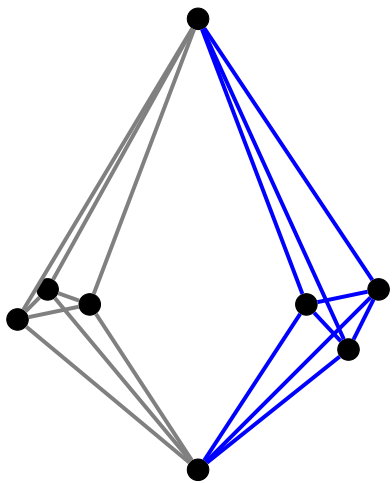

\paragraph{Molecules.}
For some classes of graphs, the 3-dimensional analogue of Theorem~\ref{thm:laman} is true.
The most interesting class appears in a statement which used to be called
the ``Molecular Conjecture'', until it was proven in \cite{KatohTanigawa:11}. It is of special interest
because it makes a statement on linkages that appear as models of molecules: atoms are modeled as balls with cylinders
attached. A molecular joint is a cylinder who is joined to an atom at both of its ends (see Figure~\ref{fig:molecule}).
From a kinematic point of view, a molecule model is a linkage with R-joints, such that for each link, all axis
of joints attached to this link meet in a fixed point (the center of the atom).

\begin{figure}[h]
\begin{center}
\includegraphics[height=4cm]{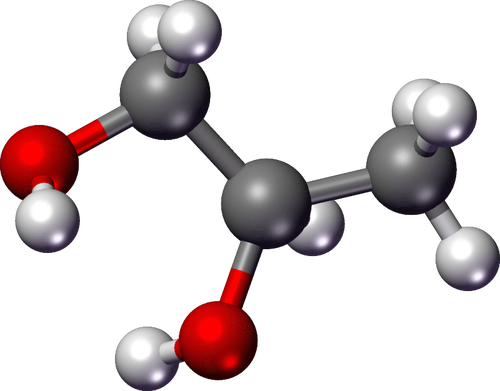}
\end{center}
\caption{A kinematic model of the Methoxyethanol molecule $\mathrm{C}_3\mathrm{H}_6(\mathrm{OH})_2$. 
        The cylinders are joints allowing a rotation around the central axis of the cylinder. 
	Note that the axis always passes through the centers of the joined atoms.}
\label{fig:molecule}
\end{figure}

The following equivalent re-formulation appear in \cite{Jackson_Jordan:07}.
For any graph, we can define its {\em square} by drawing an edge between any two vertices of graph distance two.
A graph is called a {\em square graph} if it can be obtained as the square of a subgraph. 

\begin{thm}[KatohTanigawa]
Assume that $G=(V,E)$ is a square graph such that $|E|=3|V|-6$. Assume that $|E'|\le 3|V'|-6$ for every
subgraph of $G$. Then a generic assignment of the vertices by points in $\R^3$ defines a rigid embedding.
\end{thm}

To see the equivalence with the molecule conjecture, start with a molecule and draw a graph $G$ with vertices
corresponding to atoms and edges corresponding to cylinders in the molecular model. It is clear that every motion
of the molecule fixes the length of each edge. However, every such motion also fixes the angle between two cylinders
attached to the same atom. But this is equivalent to the statement that the motion fixes the length between the
two atoms that are on the other end of the two cylinders. If you add an edge for any two such atoms, then
you get exactly the square of $G$.

\section{Overconstrained Linkages} \label{sec:over}

Let us call a linkage {\em paradoxical} if a generic linkage with the same combinatoric structure is rigid,
but the linkage itself is moving. For instance, an instance of a Laman graph which is mobile in the plane
is paradoxical. 

\paragraph{Should we Expect Paradoxical Linkages?}
Let us do a simple variable counting, as in the CGK formula, to see
if we should be surprised by the existence of paradoxical linkages. Fix a combinatorial structure, for instance
a Laman graph $G=(V,E)$. For a generic instance, the number of non-equivalent configurations is finite. These configurations
are real solutions of a system of algebraic equations; let $N_G$ be the number of complex solutions of these system.
Note that the number of complex solutions does not depend on the choice of the generic instance, as long as the choice
is generic, in contrast to the number of real solutions, which would depend on the choice of a generic instance. 

For any system of equations that has finitely many solutions, it is possible to compute a single univariate polynomial,
such that the solutions of the system are in bijection with the zeroes of the polynomial. In theory, it is possible to
compute such a polynomial by introducing a new variable together with a generic linear equation between the new variable
and the old variables, and then by eliminating all old variables. (In practice, it turns out that the elimination is
quite costly.) The process can even be carried out in the presence of parameters, which will then also appear in the coefficients
of the univariate polynomial. Let us therefore assume that we have now, for each graph $G=(V,E)$, such a polynomial $F_G$,
with coefficients depending on an edge length assignment $\lambda$. The degree of $F_G$ would then have to be equal to $N_G$,
because it has $N_G$ complex solutions and we may assume that $F_G$ is squarefree. 

Now, a labeled graph $(V,E,\lambda)$ is mobile
if and only of all $N_G+1$ coefficients of the polynomial are zero, i.e. the polynomial $F_G$ vanishes identically and
there are infinitely many configurations. (We have to take non-real configurations into account, but let us ignore this
point for the moment.) The instances of the graph form a family of dimension $|E|$ parametrized by the edge lengths.
In order to find a paradoxical linkage, we need to find a solution of a system in $|E|$ variables with $N_G+1$
equations. So we need to compare these two numbers. If the number $|E|$ of variables is bigger than or equal to
the number $N_G+1$ of equations, then we should not be surprised by the existence of paradoxical linkages.

Currently, we do not know any lower bounds for $N_G$, but there are conjectured lower bounds which are exponential in $|E|$,
so the system of equations that would have to be fulfilled for the parameters of a paradoxical linkage would be
highly overdetermined. This is also true for small graphs: for $5\le |V|\le 12$, the numbers $N_G$ are all known
\cite{Schicho:17c}, and we always have $|E|<N_G+1$. 
Consequently, the very existence of paradoxical linkages is itself paradoxical! At least, this is so for the type of linkages we
considered in this counting, namely moving graphs in the plane.

\paragraph{Bipartite Graphs.}
The smallest mobile Laman graphs have 6 vertices. One is the complete bipartite graph $K_{3,3}$. In \cite{Dixon:99},
Dixon describes a construction to make arbitrary bipartite graphs mobile. The set $V$ of vertices is partitioned into
two disjoint subsets $V_1,V_2$. Put all vertices in $V_1$ on the $x$-axis and all vertices of $V_2$ on the $y$-axis.
An easy exercise using Pythagoras' Theorem shows that the linkage is actually moving.

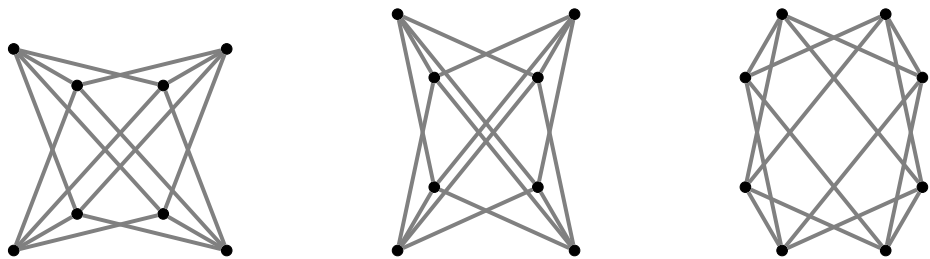
\begin{figure}[h]
\begin{center} 
\begin{tikzpicture}[scale=0.9]
      \node[vertex] (a1) at (1.573,1.490) {};
      \node[vertex] (a2) at (1.573,-1.490) {};
      \node[vertex] (a3) at (-1.573,1.490) {};
      \node[vertex] (a4) at (-1.573,-1.490) {};
      \node[vertex] (b1) at (0.636,0.949) {};
      \node[vertex] (b2) at (0.636,-0.949) {};
      \node[vertex] (b3) at (-0.636,0.949) {};
      \node[vertex] (b4) at (-0.636,-0.949) {};

      \draw[edge] (a1)edge(b1) (a1)edge(b2) (a1)edge(b3) (a1)edge(b4) 
                  (a2)edge(b1) (a2)edge(b2) (a2)edge(b3) (a2)edge(b4) 
                  (a3)edge(b1) (a3)edge(b2) (a3)edge(b3) (a3)edge(b4) 
                  (a4)edge(b1) (a4)edge(b2) (a4)edge(b3) (a4)edge(b4) ;
    \end{tikzpicture}\hspace{2cm}
\begin{tikzpicture}[scale=0.9]
      \node[vertex] (a1) at (1.307,1.747) {};
      \node[vertex] (a2) at (1.307,-1.747) {};
      \node[vertex] (a3) at (-1.307,1.747) {};
      \node[vertex] (a4) at (-1.307,-1.747) {};
      \node[vertex] (b1) at (0.765,0.810) {};
      \node[vertex] (b2) at (0.765,-0.810) {};
      \node[vertex] (b3) at (-0.765,0.810) {};
      \node[vertex] (b4) at (-0.765,-0.810) {};

      \draw[edge] (a1)edge(b1) (a1)edge(b2) (a1)edge(b3) (a1)edge(b4) 
                  (a2)edge(b1) (a2)edge(b2) (a2)edge(b3) (a2)edge(b4) 
                  (a3)edge(b1) (a3)edge(b2) (a3)edge(b3) (a3)edge(b4) 
                  (a4)edge(b1) (a4)edge(b2) (a4)edge(b3) (a4)edge(b4) ;
    \end{tikzpicture}\hspace{2cm}
\begin{tikzpicture}[scale=0.9]
      \node[vertex] (a1) at (1.307,0.810) {};
      \node[vertex] (a2) at (1.307,-0.810) {};
      \node[vertex] (a3) at (-1.307,0.810) {};
      \node[vertex] (a4) at (-1.307,-0.810) {};
      \node[vertex] (b1) at (0.765,1.747) {};
      \node[vertex] (b2) at (0.765,-1.747) {};
      \node[vertex] (b3) at (-0.765,1.747) {};
      \node[vertex] (b4) at (-0.765,-1.747) {};

      \draw[edge] (a1)edge(b1) (a1)edge(b2) (a1)edge(b3) (a1)edge(b4) 
                  (a2)edge(b1) (a2)edge(b2) (a2)edge(b3) (a2)edge(b4) 
                  (a3)edge(b1) (a3)edge(b2) (a3)edge(b3) (a3)edge(b4) 
                  (a4)edge(b1) (a4)edge(b2) (a4)edge(b3) (a4)edge(b4) ;
    \end{tikzpicture}\hspace{2cm}
\end{center}
\caption{A mobile complete bipartite graph $K_{4,4}$. Its points form two rectangles sharing their symmetry axes.}
\label{fig:dixon2}
\end{figure}

Using computer algebra, Husty/Walter\cite{HustyWalter:07} proved that Dixon's construction is one of two possible
mobile $K_{3,3}$'s; in all other cases, $K_{3,3}$ is rigid.
The second mobile $K_{3,3}$, also found in \cite{Dixon:99}, is a mobile $K_{4,4}$ with two points removed -- see 
Figure~\ref{fig:dixon2}. The configuration has a finite symmetry group, namely the symmetry of a rectangle. Indeed, the
points form two rectangles sharing their symmetry axes. 

Note that Dixon I applies to arbitrary bipartite graphs.
In contrast, the symmetric construction Dixon II does not scale, it just applies to $K_{4,4}$ and to its subgraphs.

\paragraph{NAC colorings.}
Another construction that does scale is based on the possibility of partitioning the set $E$ of edges into two non-empty subsets
$E_r,E_b$ of red and blue edges. We assume that every cycle in $G$ is either unicolored or has at least two edges of
both colors; especially, triangles are always unicolored. Such a partition is called a NAC -- ``no almost (unicolored) cycle'' --
coloring. For each connected component of the subgraph $R_i$ of $(V,E_r)$ we assign a complex number $z_i$, and for each
vertex of the subgraph $B_j$ of $(V,E_b)$, we assign a complex number $w_j$. Then we choose a real parameter $t$ parametrizing
a periodic motion, as follows: map any vertex in $R_i\cap B_j$ to the point $z_i+e^{it}w_j\in\C$. But $\C$ is a model
for the plane $\R^2$. Hence we have constructed, for any real value of $t$, a configuration of the graph in $\R^2$.
The construction is continuous in $t$, so we may call it a motion. The blue
edges always keep their orientation while the red edges are rotated with uniform speed, as in Figure~\ref{fig:nac}.

\begin{figure}[h]
\begin{center} 
\includegraphics[width=12cm]{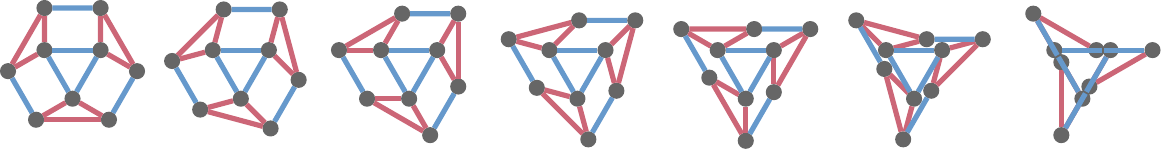}
\end{center}
\caption{A mobile graph with a NAC coloring. The blue edges remain parallel to the original orientation, and the orientation
	of the red edges rotates with speed that is independent of the edge, as long as it is red.}
\label{fig:nac}
\end{figure}

A partition of $E$ into $E_r\cup E_b$ is a NAC coloring if and only if we can map the vertices into the plane so that all red edge are
parallel to the first coordinate axis and all blue edges are parallel to the second coordinate exists. It is obvious that this map
defines a flexible embedding. Such a moving graph can be constructed by taking a very small moving graph, with three vertices and two edges,
and enlarging it by parallel copies of edges. But wait - we can do the same with other graphs as well! Let us start with 
a moving quadrilateral. Then we add more edges that are parallel to one of the four edges of the quadrilateral. We get a bigger graph
with the property that every motion of the quadrilateral induces a motion of the bigger graph -- see Figure~\ref{fig:oj} for 
an example.

\begin{figure}[h]
\begin{center} 
    \begin{tikzpicture}[scale=0.5]
      \node[vertex] (a) at (0,0) {};
      \node[vertex] (b) at (5,0) {};
      \node[vertex] (c) at (-2,1) {};
      \node[vertex] (d) at (3,1) {};
      \node[vertex] (e) at (2.411,3.766) {};
      \node[vertex] (f) at (7.411,3.766) {};
      \node[vertex] (g) at (0.411,4.766) {};
      \node[vertex] (h) at (5.411,4.766) {};

      \draw[redge] (a)edge(b) (b)edge(d) (d)edge(e) (a)edge(e);
      \draw[edge]  (a)edge(c) (b)edge(f) (c)edge(d) (e)edge(f) (f)edge(h) (d)edge(h) (c)edge(g) (e) edge (g) (g)edge(h);
    \end{tikzpicture}
    \hspace{1.5cm}
    \begin{tikzpicture}[scale=0.5]
      \node[vertex] (a) at (0,0) {};
      \node[vertex] (b) at (5,0) {};
      \node[vertex] (c) at (-1,2) {};
      \node[vertex] (d) at (4,2) {};
      \node[vertex] (e) at (2,4) {};
      \node[vertex] (f) at (7,4) {};
      \node[vertex] (g) at (1,6) {};
      \node[vertex] (h) at (6,6) {};

      \draw[redge] (a)edge(b) (b)edge(d) (d)edge(e) (a)edge(e);
      \draw[edge]  (a)edge(c) (b)edge(f) (c)edge(d) (e)edge(f) (f)edge(h) (d)edge(h) (c)edge(g) (e) edge (g) (g)edge(h);
    \end{tikzpicture}
\end{center}
\caption{A moving Laman graph with 8 vertices and 13 edges.
	The two figures -- 2D, not 3D! To see this picture correctly, please switch off your spatial perception for a moment! 
	-- show two of infinitely many possible configurations of the graph in $\R^2$, with the same edge lengths.
	Every edge is parallel to one of the four sides of the red quadrilateral. 
	The red quadrilateral has obviously infinitely many configurations; and any configuration of the
	red quadrilateral can be extended to a configuration of the whole graph.}
\label{fig:oj}
\end{figure}
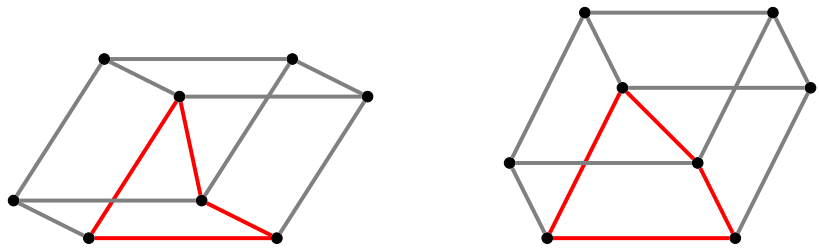

%The NAC construction of a mobile graph maps all points in the intersection of a red and of a blue components to the same
%point in the plane. In other words: if there is a pair of distinct vertices $u,v\in V$ that can connected by a red path
%and by a blue path, then these two vertices are identified.

In Section~\ref{sec:bond}, we will see that the existence of a NAC coloring is not only sufficient, but also necessary for
the existence of a length assignment that makes a given graph mobile in $\R^2$. This result requires a few tools from algebraic geometry.
More examples of graphs moving in the plane and NAC-colorings can be found in
\url{https://jan.legersky.cz/project/movablegraphs/}.

\section{Revolute Loops and Dual Quaternions} \label{sec:rl}

Let $n\ge 4$. An $n$R chain is a linkage consisting of $n+1$ links connected by $n$ revolute joints. In robotics, the first link is
called the {\em base} and the last link is called the {\em hand} or {\em end effector}. Each joint can is controlled by
an electric motor in such a way that the end effector performs a particular task. 

If we firmly connect the first and the last link of an $n$R chain, then we get an $n$R loop: a linkage with $n$ links connected
cyclically by $n$ revolute joints. According to the CGK formula, the mobility is $\max(0,n-6)$. If $n\ge 7$, then a generic $n$R loop
is generically mobile. A generic 6R linkage is rigid; the number of configurations, including complex solutions, is $16$ (see
\cite[11.5.1]{selig05}). For $n=5$ and $n=4$, we obtain an overdetermined system of equations.

\begin{figure}[h]
\begin{center}
\includegraphics[height=4.5cm]{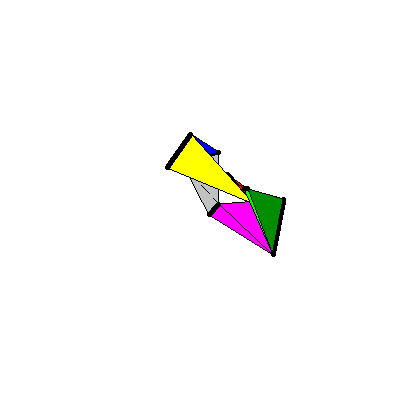} \hspace{-1.5cm}
\includegraphics[height=4.5cm]{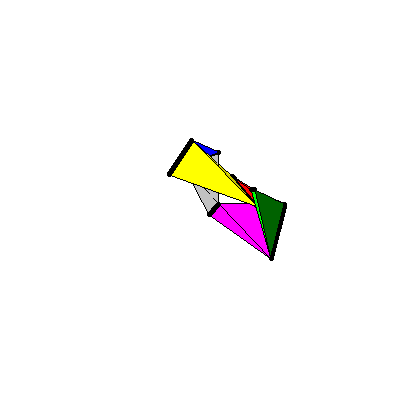} \hspace{-1.5cm}
\includegraphics[height=4.5cm]{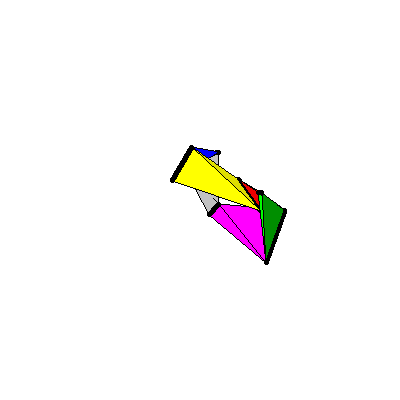} \hspace{-1.5cm}
\includegraphics[height=4.5cm]{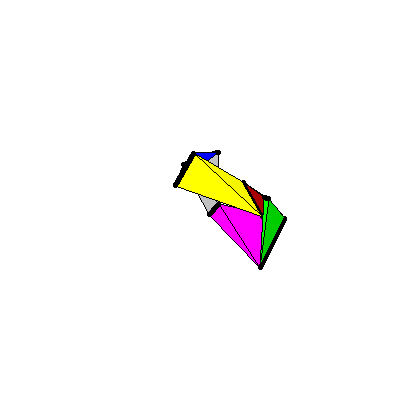} \\ \vspace{-1.5cm}
\includegraphics[height=4.5cm]{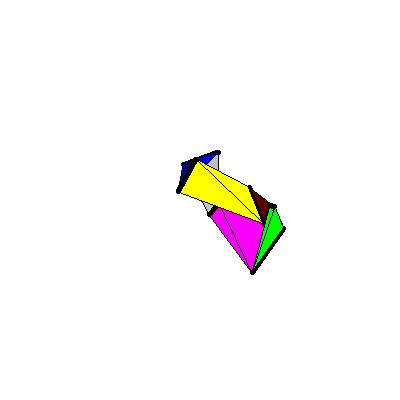} \hspace{-1.5cm}
\includegraphics[height=4.5cm]{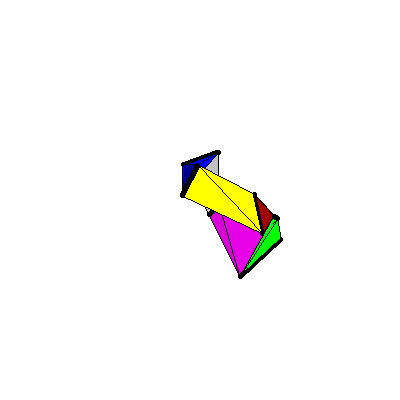} \hspace{-1.5cm}
\includegraphics[height=4.5cm]{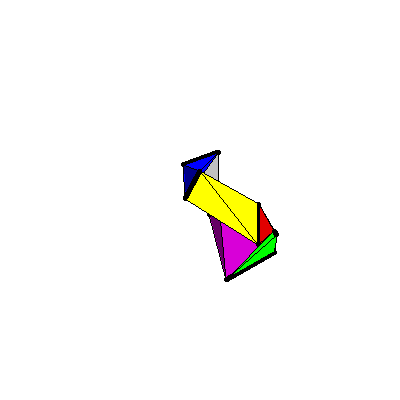} \hspace{-1.5cm}
\includegraphics[height=4.5cm]{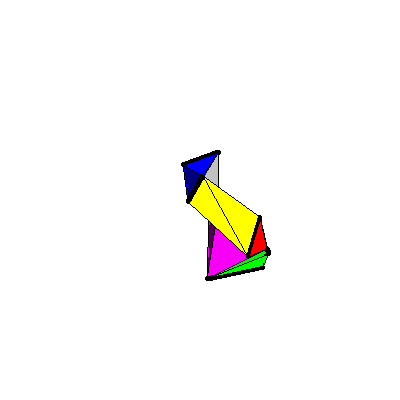} 
\end{center}
\caption{
	A thumbnail movie of a mobile 6R loop. Each of the 6 link is realized as a tetrahedron. Each tetrahedron has two edges,
	opposite to each other, playing the role of R-joints connecting the link to its to two neighbors.
	}
\label{fig:thumb}
\end{figure}

\begin{rem} \label{rem:red}
Revolute loops may be considered as special cases of linkages of graph type, in the following way: we pick two points on each joint
axis and connect them by an edge. For each link, we draw 4 additional edges connecting the points on the two axes that belong
to the link, so that every link carries a complete graph $K_4$, which is geometrically a tetrahedron.
This graph has $2n$ vertices and $5n$ edges, and it is apparent
that the linkage of graph type has exactly the same mobility as the revolute loop. See Figure~\ref{fig:thumb} for an example
of a tetrahedral 6R loop.

But even though revolute loops may be considered as a subclass of linkages of graph type, it is advantageous to introduce new
techniques especially suited for them.
\end{rem}

\paragraph{4R Loops.}
The classification of mobile 4R loops is due to Delassus~\cite{Delassus}. He proved that there are three types of mobile 4R linkages:
\begin{description}
\item[planar:] all rotation axes are parallel. Essentially, this is a quadrilateral moving in the plane. The third coordinate is not changed
	in any of the moving links.
\item[spherical:] all rotation axes pass through a single point. Essentially, this is a moving spherical quadrilateral. The planar
	case may be considered as a limit case of the spherical case.
\item[skew isogram:] Bennett ~\cite{Bennett:14} discovered a mobile 4R linkage such that the axes of joints attached to the same link are
	skew, for all four links. We describe it below in more detail.
\end{description}

Let $L_1,\dots,L_n=L_0$ be the rotation axes of in some configuration of an $n$R loop. For $i=0,\dots,n-1$, we assume that the lines $L_i$
and $L_{i+1}$ belong to the $i$-th link. Since the link is assumed to be a rigid body, the normal distance $d_i$ and the angle $\alpha_i$ 
between $L_i$ and $L_{i+1}$ does not change as the linkage moves: they are invariant parameters. Assume that none of the angles is zero, i.e., 
$L_i$ and $L_{i+1}$ are not parallel. Then there is a unique line $N_i$ intersecting both $L_i$ and $L_{i+1}$ at a right angle. 
The distance $s_i$ between $N_i\cap L_i$ and $N_i\cap L_{i+1}$ is called offset. The angles, normal distances, and offsets are $3n$
invariant geometric parameters of the linkage; in robotics, they are called the {\em invariant Denavit-Hartenberg parameters}
\cite{DeHa}.
A configuration is determined by $n$ angles, and the $3n$ invariant Denavit-Hartenberg parameters together with the $n$ configuration
parameters determine the positions of the $n$ rotation axes and the position of the links uniquely up to $\SE_3$.
These $4n$ parameters fulfill a condition of codimension~6, called the {\em closure equation}: we attach an internal
coordinate system to each link, with the axis $L_i$ being the $x$ and the common normal $N_i$ being the $z$-axis. Then the transformation 
of the $i$-th coordinate system to the $(i+1)$-th coordinate system is the composition of the translation by a vector of length $d_i$ parallel
to the $z$-axis, the rotation around the $z$-axis by the angle $\alpha_i$, the translation by a vector of length $s_i$ parallel to the
$x$ axis, and a rotation around the $x$-axis determined by the $i$-th configuration parameter. The product of all these $4n$ direct
isometries is equal to the identity, and this statement gives the closure equation.

\begin{figure}[h]
\begin{center}
\includegraphics[height=4cm]{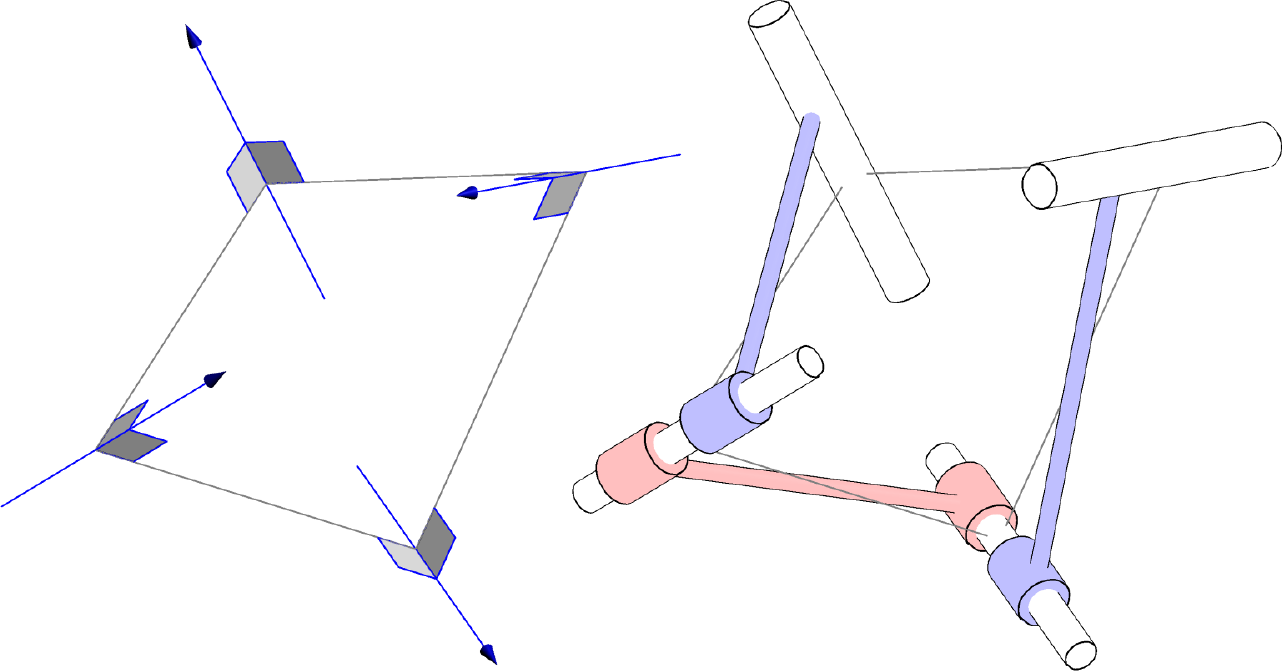}
\end{center}
\caption{The {\em skew isogram} is a mobile linkage of type 4R-loop with four rotation axes, so that axes in the same link are
	always skew. It is the only mobile 4R-loop which is neither planar (all axes are parallel) nor spherical (all axes are concurrent). 
	}
\label{fig:skewiso}
\end{figure}

A skew isogram is a 4R linkage such that the invariant Denavit-Hartenberg parameters $d_0,\dots,s_3$ satisfy the conditions
\begin{equation} \label{eq:bennett}
  d_1=d_3,d_0=d_2,\alpha_1=\alpha_3,\alpha_0=\alpha_2,\frac{d_1}{\sin(\alpha_1)}=\frac{d_0}{\sin(\alpha_0)}, s_0=s_1=s_2=s_3=0 .
\end{equation}

\paragraph{Dual Quaternions.}
In order to prove that the skew isogram is mobile, we use an algebraic way suggested by \cite{Study:03} to parametrize $\SE_3$.
The algebra $\D\H$ of dual quaternions is the 8-dimensional real vector space generated by $1,\qi,\qj,\qk,\eps,\eps\qi,\eps\qj,\eps\qk$.
Its multiplication is $\R$-linear, associative, the element $\eps$ -- the dual unit -- is central and satisfies $\eps^2=0$.
The symbols $\qi,\qj,\qk$ are Hamiltonian quaternions: $\qi^2=\qj^2=\qk^2=\qi\qj\qk=-1$. The center, generated by $1$ and $\eps$,
is called the algebra of dual numbers. Conjugation is a $\D$-linear antihomomorphism from $\D\H$ to itself: it maps $1$ to itself,
$\qi$ to $-\qi$, $\qj$ to $-\qj$, and $\qk$ to $-\qk$. For any dual quaternion $h\in\D\H$, the element $N(h):=h\overline{h}=\overline{h}h$
is a dual number, called the norm of $h$. The norm is a semigroup homomorphism with respect to multiplication; its image is the
subsemigroup consisting of all dual numbers with nonnegative primal part.

The set $\S$ of dual quaternions with norm in $\R^\ast$ is a multiplicative group. Its center is $\R^\ast$. The 
quotient group $\S/\R^\ast$ happens to be isomorphic to $\SE_3$. The isomorphism is determined by the action of
$\S/\R^\ast$ on $\R^3$. We may regard $\R^3$ as the abelian normal subgroup $T$ of $\S/\R^\ast$ of classes represented by
dual quaternions of the form $1+x\eps\qi+y\eps\qj+z\eps\qk$ (this subgroup is going to be the subgroup of translations
in $\SE_3$). The substitution of $\eps$ by $-\eps$ is an outer automorphism of $\S/\R^\ast$ of order~2 -- lets call it $\tau$ -- 
which fulfills the following property: 
if $h\in \S/\R^\ast$, then $h^{-1}\tau(h)\in T$. This implies that for all $h\in \S/\R^\ast$ and $v\in T$,
the element $h^{-1}v\tau(h)=(h^{-1}vh)(h^{-1}\tau(h))$ is in $T$, and this defines a right action of $\S/\R^\ast$ on $T$.
The bijections of $T$ in the image of this action are direct isometries, and this
defines a group isomorphism $\S/\R^\ast\cong\SE_3$. At the same time, we have constructed an embedding of $\SE_3$ into the projective space
$\P(\D\H)\cong\P^7$, as the subset defined by a quadratic form $S=0$, namely the dual part of the norm, and by a quadratic
inequation $N\ne 0$, namely the primal part of the norm.

There is a bijection between elements of order~2 in $\SE_3$ and lines in $\R^3$: every line corresponds to a {\em half turn} round that
line (a rotation by the angle $\pi$).
A point in $\SE_3\subset\P(\D\H)$ has order~2 if and only if its scalar part is zero. Here we have two linear equations,
namely the coefficient of $1$ and the coefficient of $\eps$, defining a $\P^5$ in $\P(\D\H)$. The intersection of this $\P^5$
with the quadric hypersurface defined by $S$ (a.k.a. the {\em Study quadric}) is isomorphic to the Pl\"ucker quadric, 
and the remaining six coefficients are the Pl\"ucker coordinates of lines.

Let $l\in\D\H$ be a dual quaternion representing an element of order~2 in $\SE_3$. Then $l^2=-N(l)$ is a negative real number; 
without loss of generality, we may assume $l^2=-1$. The line connecting $[1]$ and $[l]$ is contained in the Study quadric: its
elements are the rotations around the line $L$ corresponding to $l$. (Note that $[1]$ denotes the equivalence class of the
dual quaternion $1$ in $\P^7$ and does not indicate a reference to the bibliography.) 
These elements form a group; indeed, the vector space generated
by $1$ and $l$ is a subalgebra isomorphic to $\C$ over $\R$, and the projectivization of this two-dimensional real algebra is a Lie group
isomorphic to $\mathrm{SO}_2$. We call this group the {\em revolution} with axes $L$. A parametric representation of the revolution
is $(t+l)_t$, where the parameter $t$ ranges over the real projective line; the parameter $t=0$ corresponds to $[l]$, and the parameter $t=\infty$ 
corresponds to $[1]$. In general, the parameter $t$ corresponds to the cotangent of half of the rotation angle.

\begin{rem} \label{rem:revolution}
Conversely, assume that we have a line in $S$ passing through $[1]$. Then we can parametrize it by a linear polynomial in $t$ with
leading coefficient $1$, i.e., by a polynomial $(t+h)$ with $h\in S$. Because $N(t+h)=t^2+(h+\bar{h})t+N(h)$ has to be real for all $t\in\R$,
it follows that $h+\bar{h}\in\R$: the scalar part of $h$ is real (its dual part is zero). Then a reparametrization of the line 
is $(s+\frac{h-\bar{h}}{2})_s$, setting $s=t+\frac{h+\bar{h}}{2}$. This reparametrization shows that the line parametrizes a revolution
with axis corresponding to $[h-\bar{h}]$, except in the case when $N(h-\bar{h})=0$. In the exceptional case, the line will parametrize
a translation along a fixed direction.
\end{rem}

Let us now study conics passing through $[1]$ and contained in the Study quadric. Any such conic has a quadratic parametrization
$(t^2+at+b)_t$ where $a,b\in\D\H$. Does this quadric polynomial factor into two linear polynomials? And if yes, do the linear polynomials
parametrize revolutions? To answer these questions, we study 
$\D\H[t]$, the non-commutative algebra of univariate polynomials with coefficients in $\D\H$, where the variable
$t$ is supposed to be central, i.e., it commutes with the coefficients. 

\paragraph{Quaternion Polynomials.}
As a preparation, let us ask the analogous question for the non-commutative algebra $\H[t]$. We will show that
here, every polynomial can be written as a product of linear factors; in other words, the skew field of quaternions
is algebraically closed! The proof is taken from \cite{GordonMotzkin:65}.

\begin{lem}[polynomial division]\label{lem:euclid}
Let $A,B\in\H[t]$, $B\ne 0$. Then there exist unique polynomials $Q,R\in\H[t]$ such that $A=QB+R$ and either $\deg(R)<\deg(B)$ or $R=0$.
\end{lem}

\begin{proof}
We start with uniqueness. Assume $Q_1B+R_1=Q_2B+R_2$ for $\deg(R_1)<\deg(B)$ and $\deg(B_2)<\deg(R)$. We obtain $(Q_1-Q_2)B=R_2-R_1$. 
If the left side of this equation is not zero, then its degree is at least $\deg(B)$.
If the right side s not zero, then its degree is less than $\deg(B)$. Hence both sides must be zero.

For the existence, we proceed by induction on the degree of $A$: if $\deg(A)<\deg(B)$, then we set $Q:=0$ and $R:=A$. 
If $\deg(A)\ge\deg(B)$, then we can write $A=ht^{\deg(A)-\deg(B)}B+A'$ for
a suitable $h\in\H$ and $A'\in\H[t]$ with $\deg(A')<\deg(A)$. By induction, we get $A'=Q'B+R'$. 
But then we can set $Q:=Q+ht^{\deg(A)-\deg(B)}$ and $R:=R'$.
\end{proof}

If $\deg(B)=1$ in the Lemma~\ref{lem:euclid}, say $B=t-h$, then $R$ is a constant in $\H$. The constant is zero if and only
if $(t-h)$ is a right factor of $A$. If this is true, then we also say ``$h$ is a right zero of $A$''.
So, the questions is: does every polynomial $A$ of positive degree have a right zero? And maybe we are also
interested in the question how to find it.

A right zero of $A$ is also a right zero of the norm polynomial $N(A)=\bar{A}A$. We know that that the norm polynomial
is in $\R[t]$. It is also the sum of four squares -- if $A=A_0+A_1\qi+A_2\qj+A_3q_k$, then $N(A)=A_0^2+A_1^2+A_2^2+A_4^2$.
If $N(A)$ has a real zero $r$, then this real zero is also a zero of $A_0,A_2,A_2,A_3$; hence it is a zero of $A$, and we have found what we 
wanted to find.

What do we do if $N(A)$ has no real zeroes? In this case, we choose a quadratic irreducible factor $M\in\R[t]$.
By Lemma~\ref{lem:euclid}, there are $Q,R\in\H[t]$, with $\deg(R)<2$ or $R=0$, such that $A=QM+R$. We distinguish three cases.

\begin{enumerate}
\item If $R=0$, then $M$ is a right factor of $A$. Every right zero of $M$ is also a right zero of $A$. So it suffices to show
	that $M$ has a right zero. But we know that $M$ has a complex zero. So, assume that $z=a+\ci b$ is a complex zero of $M$,
	for some $a,b\in\R$, $b\ne 0$. Then we have the equation $M=(t-a-\ci b)(t-a+\ci b)$ between complex polynomials.
	But now we can replace the complex number $\ci$ by the dual quaternion $\qi$, which also fulfills the equation $\qi^2+1=0$
	It follows that $M=(t-a-\qi b)(t-a+\qi b)=0$, and $a-\qi b$ is a right zero of $M$ and also a right zero of $A$.
\item If $\deg(R)=1$, say $R=ut+v$ for suitable $u,v\in\H$, $u\ne 0$, then $h:=u^{-1}v$ is a right zero of $R$. Since
	\begin{equation} \label{eq:r}
	\overline{R}{R}=(\overline{P}-\overline{Q}M)(P-QM)=N(P)+M(-\overline{Q}P-\overline{P}Q+\overline{Q}QM) 
	\end{equation}
	is a multiple of $M$, and $\deg(\overline{R}{R})=\deg(M)=2$,
	it follows that $M$ is a left multiple of $R$. It follows that $h$ is right zero of $M$. Hence it is also a right zero of $A=QM+R$.
\item If $\deg(R)=0$, then Equation~\ref{eq:r} is self-contradictory: the right side is a multiple of $M$, and the left side
	is a nonzero constant. So, this case cannot occur.
\end{enumerate}

\begin{thm}
Every polynomial in $\H[t]$ can be written as a product of linear polynomials.
\end{thm}

The proof is already clear: given $A$ of positive degree, we can find a right $h$, write $A=A'(t-h)$, and iterate.

How many distinct factorizations do there exist? Starting with one factorization, we may get infinitely many distinct factorizations
by multiplying with constants and their inverses in between the linear factors. In order to get rid of these ``essentially same''
factorizations, it suffices to assume that the polynomial $A$ and the linear factors are normed, i.e., they have leading
coefficient~1. 

If $A$ is a multiple of an irreducible real quadric $M$ (the first case in the above case
distinction), then $A$ has infinitely many right zeroes (see \cite{GordonMotzkin:65}). But if not, then the number of distinct
factorizations is finite. Indeed, the only non-deterministic step in the iterative procedure sketched above is the choice of the
sequence of irreducible factors used for factoring out the right zeroes. In particular, we have:

\begin{prop}
A normed polynomial of degree $d$ with generic coefficients has exactly $d!$ distinct factorizations into normed linear factors.
\end{prop}

The comparison with polynomial factorization in $\C[t]$ is illuminating: there, the factorization is unique. But if we consider
two factorizations which differ only by the order of the factors as being distinct, then we have again $d!$ distinct factorizations.
In the case of $\H[t]$, permutation of factors would not lead to the same product, because $\H[t]$ is not commutative; hence
permutation is not a method to get more factorizations, and all $d!$ factorizations are different.

\paragraph{Mobility of the Skew Isogram.}
Feeling well prepared? Then, let us go back to polynomials over the dual quaternions. Can we write every polynomial in $\D\H[t]$ that
parametrizes a curve in the Study quadric into a linear factors parametrizing revolutions?

Let us assume that we have given such a polynomial $P\in\D\H[t]$. 
We can try to copy the factorization strategy that worked in $\H[t]$: factorize the norm polynomial $N(P)$, choose a quadratic irreducible 
factor $M$ (lets assume that $N(P)$ has no real zeroes for now), compute the remainder of $P$ modulo $M$; if this remainder is 
a linear polynomial $R=ut+v$ for some $u,v\in\D\H$, compute a right zero $h:=u^{-1}v$, factor out $(t-h)$ from the right, and iterate. 
This is going to work for generic coefficients. Moreover, since $N(P(t_0))$ is in $\R$ (and not in $\D\setminus\R$) for all $t_0\in\R$, 
the norm polynomial $N(P)$ is in $\R[t]$. Therefore it has a factorization into irreducible factors in $M_r\in\R[t]$, $r=1,\dots,\deg(P)$.
The right factors $(t-h_r)$ produced by our strategy satisfy the equation $(t-h_r)(t+h_r)=M_r$, so by Remark~\ref{rem:revolution},
the linear factor will generically parametrize a revolution. So, at least generically, everything is fine!

The application of our strategy leads to the following characterization of skew isograms. It was first found in \cite{BHS}
by different methods.

\begin{thm} \label{thm:bhs}
For a generic conic in the Study quadric passing through $[1]$, there is a skew isogram such that the conic parametrizes the motion
of the second link. (In particular, this skew isogram is mobile.)
\end{thm}

\begin{proof}
Let $P=t^2+at+b$ be a quadratic parametrization of the conic, with $a,b\in\D\H$. The norm $N(P)$ is a real polynomial that has only
nonnegative values. By genericity, it has no double zeroes, and can be written as a product $M_1M_2$ of two distinct quadratic
irreducible factors. For $i=1,2$, we construct as above a factorization $P=(t-r_i)(t-w_i)$ such that $N(t-w_i)=M_i$. (It follows
that $N(t-r_i)=N(t-w_{2-i})$, for $i=1,2$.)

The linear polynomials $t-r_1,t-w_1,t-r_2,t-w_2$ parametrize lines on the Study quadric. Each of them corresponds to a subgroup of
rotations around a line in $\R^3$. Let $L_1,K_1,L_2,K_2$ be these four lines, respectively. We construct a mobile 4R loop as follows:
the base link contains the lines $L_1$ and $L_2$, the first link contains the lines $L_1$ and $K_1$, the second link contains the
lines $K_1$ and $K_2$, and the third link contains the lines $L_2$ and $K_2$. For each $t\in (\R\cup\{\infty\})$, we get a configuration
of the 4R loop: the relative displacement of the first link with respect to the base link is the rotation $t-r_1$, the relative
displacement of the third link with respect to the base link is the rotation $t-r_2$, the relative motion of the second link with respect
to the first link is the rotation $t-w_1$, and relative motion of the second link with respect to the third link is the rotation $t-w_2$.
The relative position of the second link with respect to the base link can be computed in two ways, via the first link or via
the third link. In both ways, the result is $(t-r_1)(t-w_1)=(t-r_2)(t-w_2)=P$.

Once the lines are constructed, it is straightforward to compute the invariant Denavit-Hartenberg parameters of the 4R loop -- we omit this
calculation. The result are exactly the equations~\ref{eq:bennett}. It follows that the 4R loop is a skew isogram.
\end{proof}

The paper \cite{BHS} also contains the converse statement: for any skew isogram, the relative motion of two links that are not connected
by a joint is parametrized by a conic curve on the Study quadric that passes through $[1]$. In \cite{hss1}, factorizations of cubic
polynomials in $\D\H[t]$ are used to construct paradoxically moving 5R loops and 6R loops.

\paragraph{Drawing Rational Curves.}
It is time to lift the veil of the mystery about the ellipse circle shown in Figure~\ref{fig:ell}. 
This example is taken from \cite{Schicho:16c}, which contains a construction of a linkage that draws a rational plane curve.
In \cite{Schicho:18a}, the construction is extended to rational space curves.
An online illustration with several examples can be found in \url{http://www.koutschan.de/data/link/}.

The ellipse with implicit equations $\frac{(x+a)^2}{a^2}+\frac{y^2}{b^2}=z=0$ has a rational parametrization
\[ (x,y,z) = p(t) := \left( \frac{-2a}{t^2+1},\frac{2bt}{t^2+1},0 \right) . \]
For any $t\in\R$, the dual quaternion $1+\eps(\frac{-a}{t^2+1}\qi+\frac{bt}{t^2+1}\qj)$ represents a translation that maps the origin to $p(t)$.
The class of a dual quaternion is not changed when we multiply it with $t^2+1$.
So we set $P:=t^2+1+\eps(-a\qi+bt\qj)$ and try to factorize. The norm polynomial is $(t^2+1)^2$, hence our only choice of an irreducible
factor is $M=t^2+1$. The remainder of $P$ modulo $M$ is $R=\eps(-a\qi+bt\qj)$. But now something is wrong: even though $R$ has a right zero,
namely $h=-\frac{b}{a}\qk$, there is no common zero $R$ and $M$ except in the case $a=\pm b$. (If $a=\pm b$, then the ellipse is
a circle, and we are not interested.) The argument we used in the quaternion case fails because $N(R)=0$.

There is a way out: instead of factorizing $P$, we can factorize $Q:=(t-\qi)P$. The displacement $[t-\qi]$ fixes the origin, hence
the displacement $[Q(t)]$ maps the origin to the point $p(t)$, just like the translation $[P(t)]$. The remainder of $Q$ modulo $M$ is
$\eps(b-a)(\qi t-\qj)$, and this time we do have a common right zero of $M$ and $R$! Any dual quaternion of the form $-k-\epsilon(c\qj+d\qj)$
is fine. For simplicity, we set $d=0$. Now we can factor $(t+k+\epsilon c\qj)$ from the right and proceed. 
The final result is
\[ Q = (t-\qk+(a/2+b/2)\eps\qj)(t-\qk+(-a/2+b/2-c)\eps\qj)(t+\qk+\epsilon c\qj)=(t-h_1)(t-h_2)(t-h_3) \]
(we leave the remaining steps as an exercise -- they are not problematic and give a unique result).

In order to construct a linkage with mobility one, we could use another factorization with a different linear factor on the left. 
But such a factorization does not exist: the norm polynomial of $Q$ is $(t^2+1)^3$, so there is no choice of choosing different factors
of the norm polynomial. We need to mix a different quadratic irreducible polynomial into our soup.

Let $d\in\R$ and define $h_0:=2\qk+d\eps\qj$. The polynomial $(t-h_0)(t-h_1)$ has exactly two factorizations -- one we know already,
the second one is $(t-h_4)(t-h_5)$, for some $h_4,h_5\in\D\H$. Then the polynomial $(t-h_5)(t-h_2)$ also has exactly two factorizations,
and we can define two more dual quaternions such that the second factorization is $(t-h_6)(t-h_7)$. Finally, let $h_8,h_9\in\D\H$ such that
$(t-h_7)(t-h_3)=(t-h_8)(t-h_9)$. The different factorizations giving the same result correspond to paths in the directed graph $G$
in Figure~\ref{fig:dia} with equal starting and ending vertex.

\begin{figure}[h]
\begin{center}
\[   \xymatrix{
     \circled{1} \ar^{t-h_1}[rr] & & \circled{2}  \ar^{t-h_2}[rr] & & \circled{3}  \ar^{t-h_3}[rr] && \circled{4} \\
     \circled{5} \ar^{t-h_4}[rr] \ar^{t-h_0}[u] && \circled{6} \ar^{t-h_6}[rr] \ar^{t-h_5}[u] & &
		\circled{7}  \ar^{t-h_8}[rr] \ar^{t-h_7}[u] && \circled{8} \ar^{t-h_9}[u]
   } \]
\caption{
This graph displayes different factorizations of equal products in the polynomial ring of dual quaternions. For any two directed
path between two vertices, the two products of the linear polynomials appearing as edge labels in each path are equal.
The dual quaternions $h_0,\dots,h_9$ are defined as follows: 
$h_0=2\qk+d\eps\qj$, $h_1=\qk-(a/2+b/2)\eps\qj$, $h_2=\qk-(-a/2+b/2-c)\eps\qj$, $h_3=-\qk-\epsilon c\qj$,
$h_4=\qk+\frac{a+b+4d}{6}\eps\qj$, $h_5=2\qk+\frac{-2a-2b+d}{3}\eps\qj$, $h_6=\qk+\frac{-11a-5b-6c+4d}{18}\eps\qj$,
$h_7=2\qk+\frac{4a-8b+12c+d}\eps\qj$, $h_8=-\qk+\frac{-8a+16b+3c-2d}{9}\eps\qj$, $h_9=2\qk+\frac{4a-8b+d}{3}\eps\qj$.
Here, $a,b,c,d$ are arbitrary real constants.
	}
\label{fig:dia}
\end{center}
\end{figure}
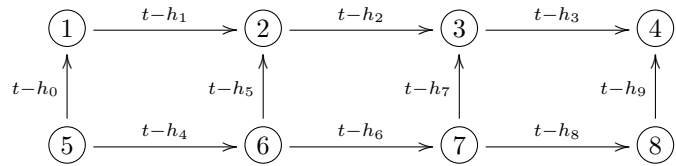

The ellipse circle consists now of eight links corresponding to the eight vertices of $G$. Two links are connected by a joint if and only
if the vertices are connected by an edge. The label of the edge -- a linear polynmial in $\D\H$ -- parametrizes the relative position of
the target link with respect to the source link. As $t$ varies, the linear polynomials parametrize a revolution. Therefore, the two links
are connected by an R-joint. Now we fix the link corresponding to vertex~4. Then the relative motion of the link corresponding to
vertex~1 maps the origin to the point $p(t)$ on the ellipse.

Note that $b=0$ is allowed; in this case, the ellipse degenerates to a line segment traced twice, and we have constructed a linkage
that draws this line.

%This has potential applications: for an arbitrary desired motion, this approach makes it possible to construct a linkage that traces this motion,
%as long as it is a rational curve in the group of direct isometries. In \cite{Schicho:16c}, a linkage is constructed that draws
%a given planar rational curve; the ellipse circle in Figure~\ref{fig:ell} is an instance of this construction. And a construction
%of a spatial linkage following a spatial motion can be found in \cite{Schicho:18a}.

\section{Symmetry} \label{sec:sym}

The second construction by Dixon of a moving $K_{4,4}$ is symmetric. Indeed, symmetry may change the counting rules and
can sometimes be the explanation of paradoxical mobility. We discuss here two cases in more detail: line symmetry and plane
symmetry. Both cases appeared in Bricard's families of moving octahedra in ~\cite{Bricard:97}. 
Schulze~\cite{Schulze:10} was the first to describe paradoxical moving symmetric graphs systematically, in every dimension.

\paragraph{Line Symmetry.} \label{ss:linesym}
We assume that we have a graph $G=(V,E)$ such that $|E|=3|V|-6$, and an assignment $(\lambda)_{e\in E}$ of a positive real
number for each edge. Generically, the configuration set, i.e., the set of all maps $V\to\R^3$ respecting edge lengths modulo $\SE_3$,
is finite: we have $3|V|-6$ variables and $|E|$ equations. Let us now assume that 
we have a graph automorphism $\tau:V\to V$ that preserves the edge assignment. Assume also that $\tau$ has order 2,
does not fix a vertex, and does not fix an edge -- a priori, an edge could be fixed if $\tau$ permutes its two vertices. 
Then $|V|$ consists of $n:=\frac{|V|}{2}$ pairs of conjugated vertices, and $E$ consists of $3n-3$ pairs of conjugated edges.
In order to construct line symmetric configurations, we fix a line $L\subset\R^3$; let $\sigma:\R^3\to\R^3$ be the rotation
around $L$ by $\pi$. For any conjugated pair $(v,\tau(v))$ of vertices, we pick one point $p_v$ anywhere in $\R^3$; the second
point is determined by $p_{\tau(v)}:=\sigma(p_v)$. The number of variables to specify all points is $3v$. There is also a two-dimensional
subgroup of $\SE_3$ fixing $L$, generated by rotations around $L$ and translations into the direction of $L$. We use two
of the variables to get a canonical representative. Hence the number of variables to specify an equivalence class of
configurations is $3n-2$. The number of equations is equal to the number of pairs of conjugated edges, which is §$3n-3$,
because conjugated edges always have the same length. Hence the expected mobility is one.

The smallest line symmetric moving graph is the 1-skeleton of an octahedron, with 6 vertices and 12 edges. The group of graph
automorphisms is isomorphic to the Euclidean symmetry group of a regular octahedron, which has 48 elements. There is a unique automorphism
of order 2 without fixed vertex and fixed edge, corresponding to the point reflection of the regular octahedron. The
construction applies, and we get a moving line symmetric octahedron (see Figure~\ref{fig:oct} left side). 
Bricard~\cite{Bricard} proved that there are three types of moving octahedra, and the line symmetric is one of the three.

\begin{figure}[h]
\begin{center} 
\includegraphics[width=5cm]{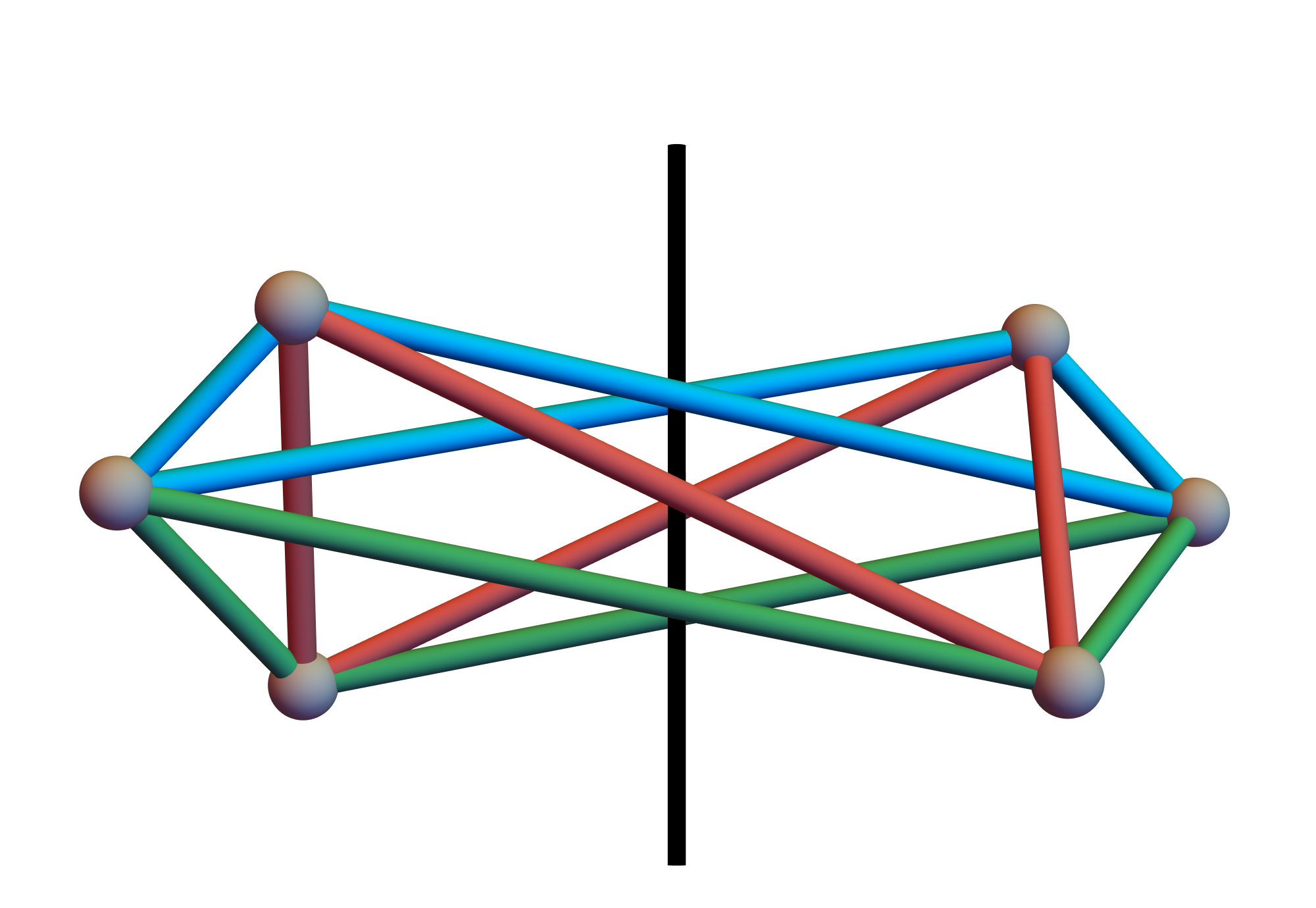}
\hspace{1cm}
\includegraphics[width=5cm]{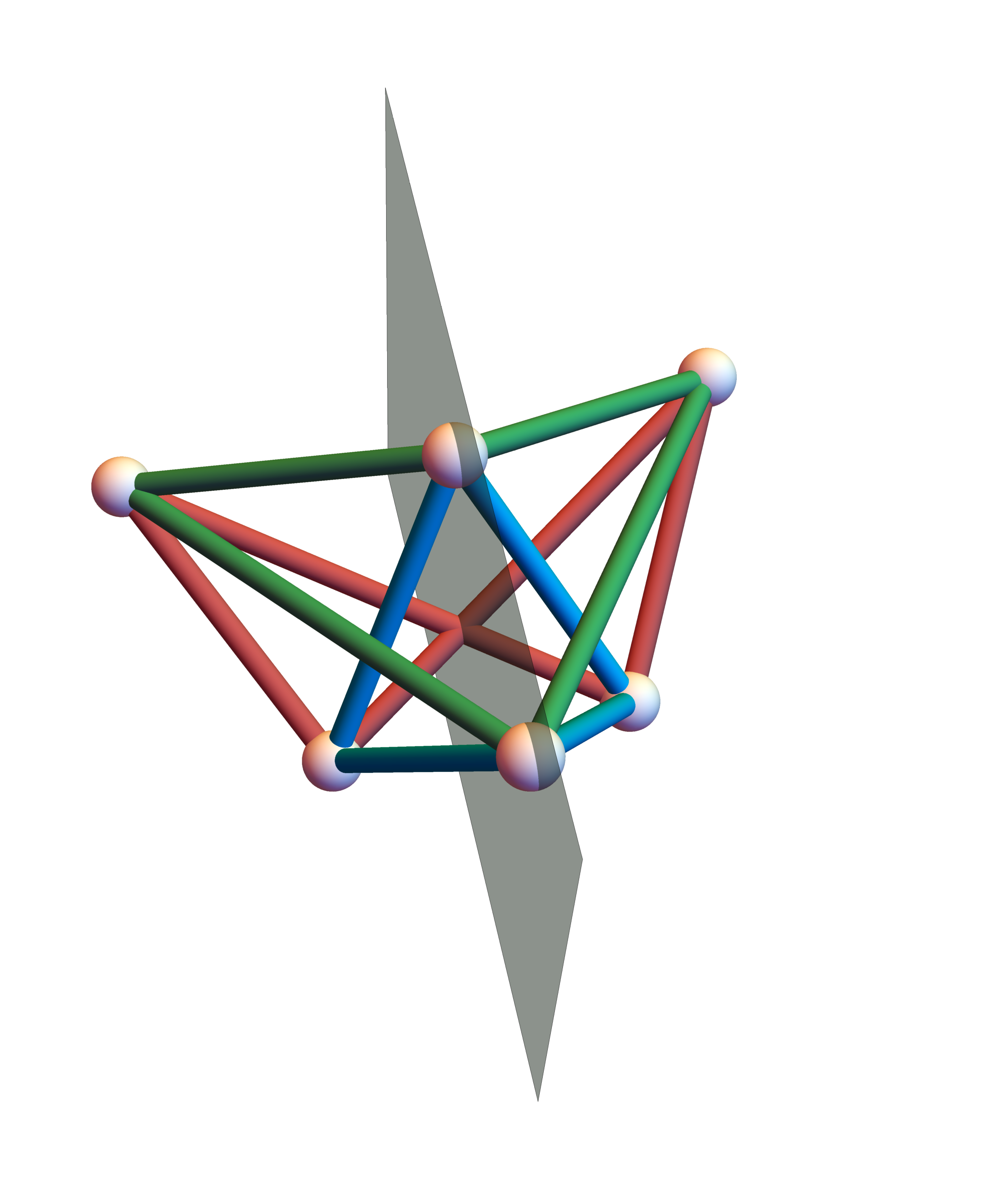}
\end{center}
\caption{Left side: a flexible octahedron that is symmetric by a line reflection; 
right side: a flexible octahedron that is symmetric by a plane reflection. 
Corresponding edges are shown in the same color.}
\label{fig:oct}
\end{figure}

More generally, we can take any centrally symmetric convex polyhedron $\Gamma$ with only triangular faces and choose as a graph 
$G=(V,E)$ its 1-skeleton. By Euler's formula, the number of edges is $3|V|-6$. The point reflection acting on $\Gamma$ defines
an automorphism of the graph which satisfies the required properties: order~2, no fixed vertex, no fixed edge. The construction
applies, and we get, for instance, a line symmetric moving icosahedron with 12 vertices and 30 edges. 

\begin{rem} \label{rem:sym}
Be careful: the point symmetry defines only the graph automorphism!
It is geometrically different from the line symmetry in all configurations we allow.
Point symmetric configurations do also exist, but only finitely many.
\end{rem}

Another classical example is Bricard's line symmetric 6R loop. Any 6R loop consists of 6 links, cyclically connected by revolute
joints that allow rotations around an axes which is common to the two attached links; generically, a 6R loop is rigid. In the line
symmetric case, we assume that the 18 invariant Denavit-Hartenberg parameters $d_0,\dots,s_5$ satisfy 
\[ d_i=d_{i+3}, \alpha_i=\alpha_{i+3}, s_i=s_{i+3} \mbox{ for } i=0,1,2 ,\]
and we are only looking for configurations such that there exists a half turn mapping the $i$-th link to the $(i+3)$-rd link.

Recall that configurations can be found by solving the closure equation (see Section~\ref{sec:rl}): 
we attach an internal coordinate system to each link and parametrize the transformation $T_i$ from the $i$-th link to the $(i+1)$-th link
(where the 6-th link is the 0-th link) by the $i$-th configuration parameter $\phi_i$. 
As mentioned above, $T_i(\phi_i)$ is the composition of the translation by a vector of length $d_i$ parallel
to the $z$-axis, the rotation around the $z$-axis by the angle $\alpha_i$, the translation by a vector of length $s_i$ parallel to the
$x$ axis, and a rotation around the $x$-axis determined by the $i$-th configuration parameter $\phi_i$.
The configuration set is the set of solutions $(\phi_0,\dots,\phi_5)$ of the closure equation 
\[ T_0(\phi_0)T_1(\phi_1)T_2(\phi_2)T_3(\phi_3)T_4(\phi_4)T_5(\phi_5) = e , \]
where $e$ is the identity of the group $\SE_3$. The functions $T_0,\dots,T_5$ depend on the invariant Denavit-Hartenberg parameters,
and as a consequence we have $T_0=T_3$, $T_1=T_4$, and $T_2=T_5$. Recall that the closure equation is a codimension~6 condition, 
because $\SE_3$ is a six-dimensional group, hence the CGK-formula estimates that there are only finitely many solutions.

Since we are only interested in line symmetric configurations, we assume $\phi_0=\phi_3$, $\phi_1=\phi_4$, and $\phi_2=\phi_5$.
The closure equation reduces to
\[ \left(T_0(\phi_0)T_1(\phi_1)T_2(\phi_2)\right)^2 = e . \]
We ignore the solutions of $T_0(\phi_0)T_1(\phi_1)T_2(\phi_2)=e$ (these are at most finitely many).
This means, we search for configuration parameters such that the transformation of the coordinate system of the 0-th link 
to the coordinate system of the 3rd link is a half turn. This is a codimension~2 condition: as we have already mentioned
in Section~\ref{sec:rl}, the set of involutions in $\SE_3$ is a 4-dimensional manifold.
Hence there is a one-dimensional set of line symmetric configurations generically.

\begin{rem} \label{rem:ls}
Is there a good reason to explain the mobility of a line symmetric linkage by the closure equation, instead of just considering them
as special cases of line symmetric linkages of graph type, as in Remark~\ref{rem:red}. Here is one: we may replace some of the
revolute joints by other types of joints, like prismatic joints, as in hydraulic telescopes, or helical joints, as commonly
seen in the form of screws. In both cases, such a joint allows a one-parameter subgroup of displacements of the connected links,
and exactly the same proof of mobility is valid. 
On the other hand, a loop with helical joints cannot be considered as a linkage of graph type, because its closure equation 
is not even algebraic.
\end{rem}

Yet another classical example, the line symmetric Stewart platform, will be explained in Section~\ref{sec:pod}. 

\paragraph{Plane Symmetry.}

Plane reflections are involutions in the group $\mathrm{E}_3$ of isometries reversing the orientation. They are of course not direct
isometries, but they still may be responsible for paradoxical
mobility of various types of linkages, similar to half turns in the case of line symmetric linkages. Let us start with 6R loops.
In a plane symmetric configuration of a 6R loop, there exists a plane reflection mapping link~0 to link~5, 
link~1 to link~4, and link~2 to link~3. The existence of a plane symmetric configuration has the following
implications on the invariant Denavit-Hartenberg parameters:
\[ d_0=d_5, d_1=d_4, d_2=d_3, \alpha_0=\alpha_5, \alpha_1=\alpha_4, \alpha_2=\alpha_3, s_1=-s_0, s_2=-s_5, s_0=s_3=0 . \]
The relations between the functions in the closure equations are the following:
\[ RT_0(\phi_0)R=T_0(-\phi_0), RT_1(\phi_1)R=T_5(-\phi_1), RT_2(\phi_2)R=T_4(-\phi_2), RT_3(\phi_3)R=T_3(-\phi_3), \]
where $R$ is the reflection by the coordinate plane $\Pi$ spanned by the first and second axes.
Instead of solving the closure equation, we find all quadruples $(\phi_0,\phi_1,\phi_2,\phi_3)$ such that
$RXR=X$, where $X:=T_0(\phi_0)T_1(\phi_1)T_2(\phi_2)T_3(\phi_3)$. An element $X\in\SE_3$ fulfills the equation $RXR=X$
if and only if it is a rotation with an axis orthogonal to $\Pi$ or a translation by a vector in $\Pi$. These rotations and translations
form a manifold of dimension~3 (isomorphic to $\SE_2$), hence the condition above is
a codimension $6-3=3$ condition. In general, there is a one-dimensional set of solutions.

For every solution $(\phi_0,\phi_1,\phi_2,\phi_3)$ of $RXR=X$, the six-tuple $(2\phi_0,\phi_1,\phi_2,2\phi_3,\phi_2,\phi_1)$ is a
solution of the closure equation:
\[ T_0(2\phi_0)T_1(\phi_1)T_2(\phi_2)T_3(2\phi_3)T_4(-\phi_2)T_5(-\phi_1) =T_0(\phi_0)XT_3(\phi_3)RT_2(-\phi_2)T_1(-\phi_1)R = \]
\[  RT_0(-\phi_0)RXRT_3(-\phi_3)T_2(-\phi_2)T_1(-\phi_1)R = RT_0(-\phi_0)RXRX^{-1}T_0(\phi_0)R = e . \]
Hence we again get a mobile 6R loop, also known as ``Bricard's plane symmetric 6R linkage''.

\begin{rem} \label{rem:ps}
As in Remark~\ref{rem:ls}, we may replace some of the revolute joints by prismatic or helical joints -- see \cite{Baker:97}.
Care has to be taken for the special role of the 0-th joint and the 3rd joint, because these two joints are supposed
to be mapped to their own inverse by the plane symmetry. This is not possible at all for helical joints. Prismatic joints
are fine, but the direction vector has to be perpendicular to the symmetry plane and not parallel to it.
\end{rem}

For linkages of graph type, there is also a construction of plane symmetric linkages that are paradoxically mobile.
We assume that we have a graph $(V,E)$ such that is generically rigid and satisfies $|E|=3|V|-6$,
for instance the 1-skeleton of a convex polyhedron with triangular faces. Assume that we have a graph automorphism $\tau:V\to V$ of order~2
that fixes $2m$ vertices and $2m-2$ edges, for some $m\ge 1$.
Choose a generic edge assignment that respects the involutive symmetry. Fix a plane $\Pi$ in $\R^3$, and let $R:\R^3\to\R^3$
be the reflection at $\Pi$.
A configuration $(p_v)_{v\in V}$ is symmetric with respect to the plane $\Pi$ if and only if $R(p_v)=p_{\tau(v)}$ holds
for all $v\in V$. 
The number of indeterminates is $3\frac{|V|-2m}{2}+4m-3=\frac{3}{2}|V|+m-3$: for each 2-orbit in $V$, the realization 
is determined by 3 indeterminates, and for each fixed point, we have two indeterminates because the point must lie in $\Pi$. 
The symmetry group of the plane has dimension~3, which reduces the number of indeterminates of equivalence classes by 3. 
The number of equations is $\frac{|E|-2m+2}{2}+2m-2=\frac{3|V|-2m-4}{2}+2m-2=\frac{3}{2}|V|+m-4$. Again, we obtain a paradoxically mobile graph.

So, how do we find graphs with an automorphism of order~2 fixing $2m$ vertices and $2m-2$ edges? Say, the graph is the
1-skeleton of a convex polyhedron $\Gamma$ with triangular faces. If $\Gamma$ is symmetric with respect to the half turn around
a line passing through 2 vertices, then we get a involution with 2 fixed points and no fixed edge, so that $m=1$. This works,
for example, for the octahedron -- see Figure~\ref{fig:oct} right side -- and for the icosahedron. 

\begin{rem}
In the construction above, there are two geometric symmetries playing entirely different roles: The line symmetry of the convex polyhedron 
defines a graph automorphism of order~2 with the right properties; the plane symmetry defines a condition on the configurations that
we consider. See also Remark~\ref{rem:sym}.
\end{rem}

Here is an example of a generically rigid graph with 12 vertices and 30 edges and with an automorphism of order~2 that fixes
$4$ vertices and $2$ edges: take a 6R loop and construct a graph as in Remark~\ref{rem:red}, by putting two vertices on each of the four rotation axes.
In this case, the plane symmetric construction just reconstructs plane symmetric 6R loops, which we already did in another way.

\section{Multipods and Group-Leg Duality}
\label{sec:pod}

The Prix Vaillant 1904 asked for curves in the Lie group $\mathrm{SE}_3$ of direct isometries such that ``many'' points in $\R^3$
move on spheres. Connecting the moving points by sticks with the midpoints of these spheres, we obtain a {\em multipod},
also known as {\em Stewart platform}, which is a linkage consisting of a fixed base and a moving platform that are connected by {\em legs} 
of fixed length that are attached to platform and base by spherical joints (see Figure~\ref{fig:planar}).
Flight simulators or other linkages that are supposed to make irregular motions are often manufactured as hexapods with additional prismatic joints
at each leg that change its length; in this section, as already stated, the leg lengths remain constant.
Each leg gives a codimension~1 condition on the displacement
of the platform with respect to the base, hence the CGK-formula gives the estimate $6-n$ for the mobility an $n$-pod. Strictly speaking, each leg may
be considered as a link that may also revolve around the line connecting its two anchor points, but we disregard this component of
the motion. 
So, pentapods are generically mobile, and hexapods are generically rigid. 

A displacement $\R^3\to\R^3$ of the platform relative to the base is given by an orthogonal matrix $M\in\mathrm{SO}_3$ 
with determinant $1$ and the image $y\in\R^3$ of the origin of the base. We set $x:=-M^ty=-M^{-1}y$ to be the preimage 
of the origin of the platform and $r:=\langle x,x\rangle = \langle y,y\rangle$, where $\langle \cdot, \cdot \rangle$ 
is the Euclidean scalar product.
If we take coordinates $m_{11}, \dots, m_{33}$, $x_1, x_2, x_3$, $y_1, y_2, y_3$ and~$r$, together with a homogenizing
variable~$h$, in~$\P^{16}$, then
a direct isometry defines a point in projective space satisfying $h \neq 0$ and
\begin{equation}
\label{eq:compactification}
  \begin{gathered}
    M M^t \; = \; M^t M \; = \; h^2 \cdot \mathrm{id}_{\R^3}, \quad \mathrm{adj}(M) \; =hM^t , \\
    M^t y + h x \; = \; 0, \quad M x + h y \; = \; 0 , \\
    \langle x,x \rangle \; = \; \langle y,y \rangle \; = \; r  h, 
  \end{gathered}
\end{equation}
where $\mathrm{adj}(M)$ is the adjugate matrix. Recall that $A\cdot\mathrm{adj}(A)=\mathrm{adj}(A)\cdot A=\det(A)\cdot\mathrm{id}_{\R^3}$
for any $A\in\R^{3\times 3}$, therefore the above equations imply $\det(M)=h^3$.
The equations~\eqref{eq:compactification} define a variety~$X$ of dimension~6 and degree~40 in~$\P^{16}$, whose real points satisfying $h \neq 0$ 
are in one to one correspondence with the elements of~$\SE_3$. We call it the {\em group variety}; its projective space $\P^{16}$ containing $X$
is called {\em group space}.

Mathematically, a leg is a triple $(a,b,d)$, where $a\in\R^3$ is a point of the base, $b\in\R^3$ is a point of the platform, 
and $d\in\R$ is a positive number, the length of the leg. We define the {\em leg variety} $Y$ as the cone over the
Segre variety $\Sigma_{3,3} \cong \P^3\times\P^3$ in the projective space $\check{\P}^{16}$; recall that the Segre variety is
a subvariety of a projective space of dimension~15 and degree  ${3+3\choose 3}=20$, hence $Y$ has dimension $7$ and degree $20$. 
The values of projective coordinates of a leg $(a,b,d)$ are $u:=1$, $a_i$, $b_j$ and
$z_{ij}:=a_ib_j$ for $i=1,2,3$, and the {\em corrected leg length} $l := \langle a, a \rangle + \langle b, b \rangle - d^2$.
The projective space $\check{\P}^{16}$ containing $Y$ is called {\em leg space}. 

The reason for this very specific choice of coordinates is the following. The algebraic condition 
$\langle Ma+y-b,Ma+y-b\rangle = d^2$ is bilinear in these coordinates:
\begin{equation}
\begin{gathered}
\label{eq:bilinear_sphere_condition}
  l h + u r - 2 \sum_{i=1}^3 a_ix_i - 2 \sum_{j=1}^3 b_jy_j
 - 2 \sum_{i, j = 1}^{3} z_{ij} m_{ij} \; = \; 0.
\end{gathered}
\end{equation}
Hence it defines a duality between group space and leg space. Every point in group space, in particular every group element, corresponds
to a hyperplane in leg space; every point in leg space, in particular every leg, corresponds to a hyperplane in group space. More
generally, to every $k$-plane in group space there is a dual $(15-k)$-plane in leg space, for $k=0,\dots,15$.

The duality has various implications for multipods, whether they are paradoxical or not. To start with, choose 6 generic legs.
They span a generic 5-plane in leg space. The dual 10-plane in group space is also generic and, since it has codimension 6,
intersects $X$ in $\deg(X)=40$ points (real or complex). Hence a generic hexapod has 40 configurations, possibly complex.

Now, we choose 5 generic legs. They span a generic 4-plane in leg space, dual to a generic 11-plane in group space, which intersects
$X$ in a curve $C$ of degree~40: the configuration curve of a generic pentapod. We can compute its genus. We first compute the Hilbert
series of $X$ from a generating set of its ideal: $H(X)(t)=\frac{1+10t+18t^2+10t^3+t^4}{(1-t)^7}$. Because $C$ is a codimension 5 subvariety of $X$
defined by 5 linear forms, we may compute the Hilbert series of $C$ from the Hilbert series of $X$:
\[ H_C(t) = H_X(t)(1-t)^5 = \frac{1+10t+18t^2+10t^3+t^4}{(1-t)^2} = 1+12t+41t^2+80t^3+120t^4+\dots . \]
This implies that $C$ is a curve of genus 41, and its embedding in $\P^{11}$ is half canonical.

\paragraph{The Bricard-Borel Infinity-Pod.}
Here is the infinity-pod that has won the Prix Vallaint to Borel~\cite{Borel:08} and Bricard~\cite{Bricard:07}. 
We intersect $X$ with the 3-space defined by
\[ r+\beta h-2\alpha m_{11} = m_{11}-m_{22} = m_{12}+m_{21} = m_{33}-h = x_{3}+y_{3} = \]
\[ m_{13} = m_{23} = m_{31} = m_{32} = x_1 = x_2 = y_1 = y_3 = 0, \]
where $\alpha,\beta\in\R$ are real parameters such that $\alpha\ne 0$. The result is a quartic curve defined by the equations 
\[ m_{11}^2+m_{12}^2-h^2=x_3^2-2\alpha m_{11}h+\beta h^2 = 0 \]
and by the linear equations above. It parametrizes a motion $C$ contained in the two-dimensional stabilizer of the third axes $L$,
generated by rotations around $L$ and translations in the direction of $L$. The dual 12-plane in leg space is defined by
\[ z_{12}-z_{21} = z_{11}+z_{22}-\alpha u = a_3-b_3 = l-2z_{33}-\beta u = 0 . \]
A leg $(a,b,d)$ in the intersection with $Y$ if and only if
\[ a_1b_2-a_2b_1 = a_1b_1+a_2b_2-\alpha = a_3-b_3 = a_1^2+a_2^2+b_1^2+b_2^2-d^2-\beta = 0 . \]
For any point $(a_1,a_2,a_3)$ in the base such that $(a_1,a_2)\ne (0,0)$, there is a unique point
$(b_1,b_2,b_3)$ in the platform and a length such that the motion $C$ keeps the distance of base and platform point equal to $d$.
To get the platform point corresponding to a given base point $(a_1,a_2,a_3)$, we invert its projection $(a_1,a_2)$
on the circle with radius $\sqrt{|\alpha|}$ and keep the third coordinate; if $\alpha<0$, then we also have to rotate
the projection by an angle of $\pi$. 

In the degenerate case $\alpha=\beta=0$, one of the equations of the quartic curve is a perfect square, and the reduce equations
$m_{11}^2+m_{12}^2-h^2=x_3$ define a conic in a 2-space. In leg space, we have one less linear equation: $z_{12}-z_{21}=z_{11}+z_{22}=l-2z_{33}=0$,
or equivalently  
\[ a_1b_2-a_2b_1 = a_1b_1+a_2b_2 = a_1^2+a_2^2+b_1^2+b_2^2-d^2 = 0 . \]
Here we get a four-dimensional set of possible legs with two components, 
namely the set of legs where the platform point or the base point lies on the $z$-axis. The motion is just a revolution around
the $z$-axis.  % (see Figure~\ref{fig:butterfly}).

\paragraph{Planar Multipods.}
We consider now the linear subspace $L_p\subset\check{\P}^{16}$ of dimension~9 in the leg space defined by the equations
\[ a_3=b_3=z_{13}=z_{23}=z_{31}=z_{32} = z_{33}= 0 . \]
Its intersection $Y_p$ with the leg variety consists of all legs such that the two anchor points lie on a fixed plane. 
The variety $Y_p$ is the Segre variety $\Sigma_{2,2} \cong \P^2\times\P^2$; let us call its elements informally
{\em planar legs}. The degree of $Y_p$ is ${2+2\choose 2}=6$. 

A multipod such that all its base points are coplanar and all its platform
points are coplanar is called a {\em planar multipod} (see Figure~\ref{fig:planar}).
To obtain the configuration of a planar multipod, one has to intersect the dual space of the linear span
of all legs with the group variety $X$. The linear span of the legs is contained in $L_p$, hence the dual space of the linear span contains
the dual space $L_p^\bot$. This linear space does not intersect the group variety, otherwise we would have a displacement that
preserves the length of all legs in $Y_p$, which is impossible. What we can say is that the projection 
$\P^{16}\dashrightarrow \P^9$ with center $L_p^\bot$ projects the group variety to a subvariety $X_p\in\P^9$ of dimension~6
and degree~20 by a map that is generically 2:1. Hence the configuration of a planar multipod come in pairs: for every configuration,
there is a conjugated configuration. It can be obtained by an outer automorphism of $\SE_3$, namely the conjugation by
the reflection with respect to the plane containing the anchor points. 

\begin{figure}[h]
\begin{center} 
        \begin{overpic}[width=55mm]{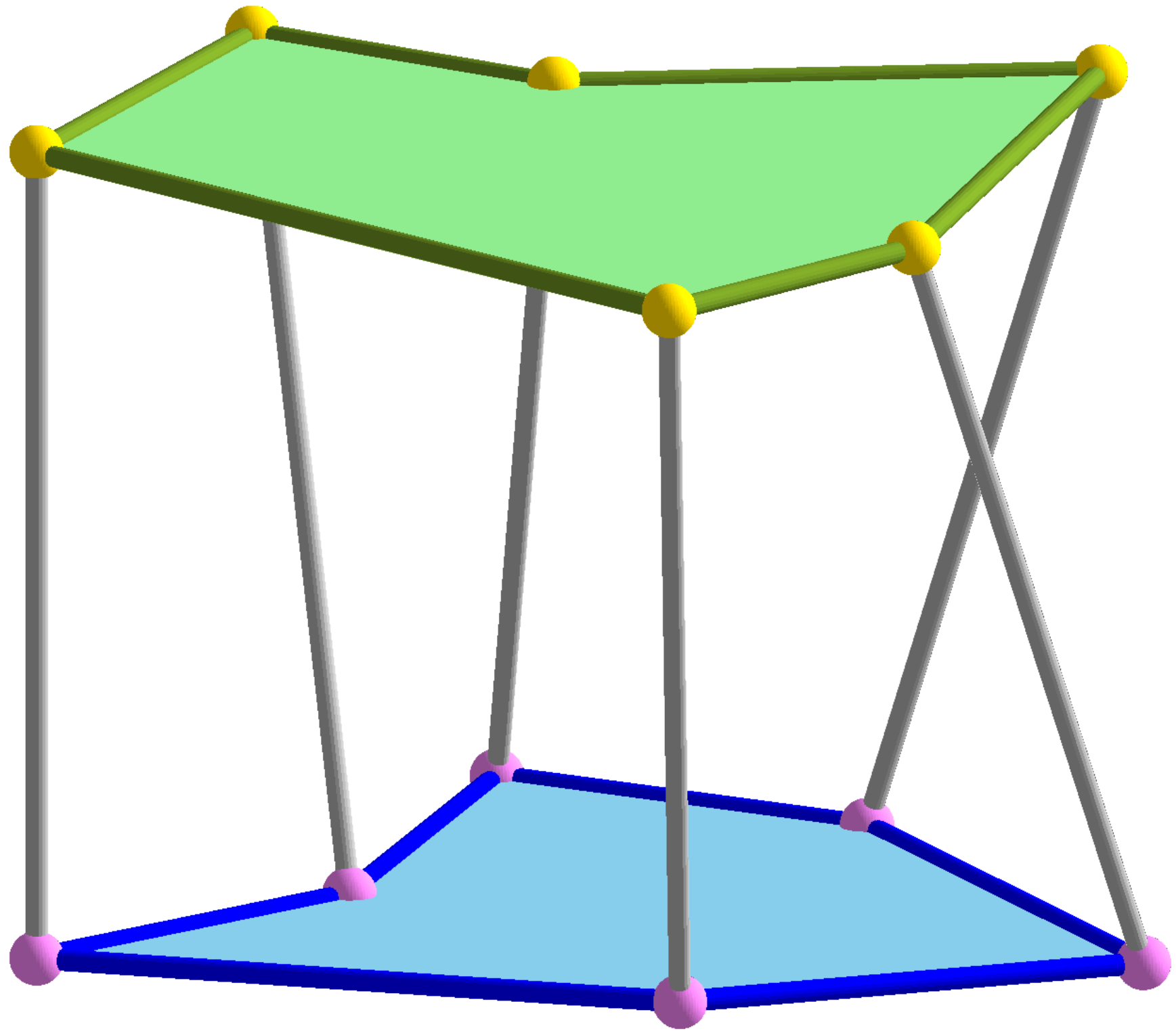}
                \begin{small}
                        \put(-4,0){$a_1$}
                        \put(50,-4){$a_2$}
                        \put(100,2){$a_3$}
                        \put(78,19){$a_4$}
                        \put(34,24){$a_5$}
                        \put(20,13){$a_6$}
                        \put(-6,73){$b_1$}
                        \put(50,56){$b_2$}
                        \put(70,69.5){$b_3$}
                        \put(98,80){$b_4$}
                        \put(45,85){$b_5$}
                        \put(12,86){$b_6$}
                \end{small}
        \end{overpic}
\end{center}
\caption{A planar hexapod. For any configuration, there is also a conjugated configuration that can be obtained by
	reflection on the green base plane.}
\label{fig:planar}
\end{figure}
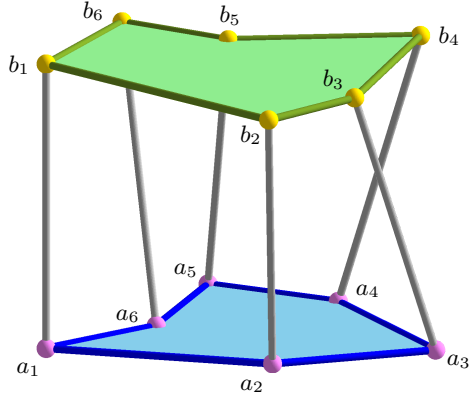

It is surprisingly easy to construct paradoxically moving planar hexapods. Here is the reason.

\begin{thm}[Duporcq]
Let $y_1,\dots,y_5\in Y_p$ be five generic planar legs. Then there exists a planar leg $y_6\in Y_p$ such that the configuration
space of the pentapod defined by $(y_1,y_2,y_3,y_4,y_5)$ is equal to the configuration space of the hexapod
defined by $(y_1,y_2,y_3,y_4,y_5,y_6)$.
\end{thm}

\begin{proof}
Let $V\subset \check{\P}^{16}$ be the linear span of $y_1,\dots,y_5$. Its dimension is $4$. The dimension of $Y_p$ is $5$.
Both $V$ and $Y_p$ are contained in $L_p\cong\P^9$, hence the intersection $Y_p\cap V$ is finite. Its cardinality is
equal to the degree of $Y_p$, which is 6. We know already 5 points; we choose $y_6$ to be the 6-th.

For both, the pentapod and the hexapod, the configuration set of the pentapod is the intersection of the group variety $X$ 
with the dual space $V^\bot$. The linear condition imposed by the 6-leg does not impose an independent condition
because it lies in the linear span of the other five.
\end{proof}

\paragraph{Line-symmetric Multipods.}
Another class of paradoxically moving hexapods is the class of line symmetric hexapods. They can be obtained as special
cases of line symmetric moving graphs (see Section~\ref{ss:linesym}). The graph consists of two octahedra $G_1,G_2$ together with
six edges each joining one point of $G_1$ to one point of $G_2$, so that these six edges provide a graph symmetry between
$G_1$ and $G_2$. The automorphism $\tau$ of the whole graph $G$ maps each vertex $v_1$ of $G_1$ to the vertex in $G_2$ 
connected with the unique vertex in $G_1$ that is not connected with $v_1$ (see Figure~\ref{fig:graphaut}). This graph automorphism does not 
fix any vertex or any edge. We fix a line $L$ of symmetry and embed $G$ so that the half turn around $L$ maps each vertex
$v$ to the image of $\tau(v)$, generically with respect to this condition. By the count in Section~\ref{ss:linesym},
the configurations are solutions of an algebraic system in 16 unknowns and 15 equations, implying mobility.

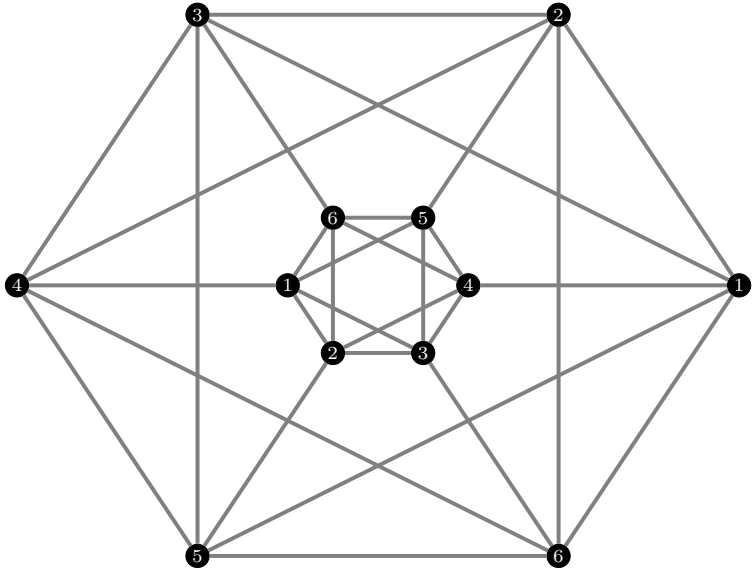
\begin{figure}[h]
\begin{center}
    \begin{tikzpicture}[scale=0.3]
      \node[lnode] (a1) at (16,0) {1};
      \node[lnode] (a2) at (8,12) {2};
      \node[lnode] (a3) at (-8,12) {3};
      \node[lnode] (a4) at (-16,0) {4};
      \node[lnode] (a5) at (-8,-12) {5};
      \node[lnode] (a6) at (8,-12) {6};

      \node[lnode] (b1) at (4,0) {4};
      \node[lnode] (b2) at (2,3) {5};
      \node[lnode] (b3) at (-2,3) {6};
      \node[lnode] (b4) at (-4,0) {1};
      \node[lnode] (b5) at (-2,-3) {2};
      \node[lnode] (b6) at (2,-3) {3};

      \draw[edge]  (a1)edge(a2) (a2)edge(a3) (a3)edge(a4) (a4)edge(a5) (a5)edge(a6) (a6)edge(a1)
                   (a1)edge(a3) (a2)edge(a4) (a3)edge(a5) (a4)edge(a6) (a5)edge(a1) (a6)edge(a2)
                   (b1)edge(b2) (b2)edge(b3) (b3)edge(b4) (b4)edge(b5) (b5)edge(b6) (b6)edge(b1)
                   (b1)edge(b3) (b2)edge(b4) (b3)edge(b5) (b4)edge(b6) (b5)edge(b1) (b6)edge(b2)
                   (b1)edge(a1) (b2)edge(a2) (b3)edge(a3) (b4)edge(a4) (b5)edge(a5) (b6)edge(a6);
    \end{tikzpicture}
\end{center}
\caption{A graph consisting of two octahedra and 6 additional edges, with a graph automorphism of order 2 that does not fix
	and vertex or any edge. The automorphism is shown by vertex orbits: conjugated vertices have equal labels. By symmetric counting of variables 
        and equations, a generic line symmetric embedding does move.
        In this motion, the two octahedra are rigid, and we obtain a moving hexapod.}
\label{fig:graphaut}
\end{figure}

It pays off to analyze the situation again by group-leg duality, following an analysis from \cite{Borel:08}.
Let $L_i\subset\P^{16}$ be the linear subspace in group
space defined by the linear equations $M=M^t$ and $x=y$; it intersects $X$ in the subset $X_i$ of all displacement of order~2 or 1.
Note that the order~2 elements in $\SE_3$ are exactly the rotations around lines by an angle of $\pi$.
These are six equations, hence $\dim(L_i)=10$. 
The dual subspace $L_i^\bot$ in leg space has dimension~5 and is defined by the equations 
$l=u=z_ii=a_i+b_j=z_{ij}=0$ for $i,j=1,\dots,3$, $i\ne j$.
We have a situation that mirror the planar hexapod case: the subspace $L_i^\bot$ does not intersect the leg variety, otherwise
there would be a leg which does not change length in all involutions. But the projection
$\check{\P}^{16}\dashrightarrow \check{\P}^{10}$ with center $L_i^\bot$ projects the leg variety $Y$ to a subvariety $Y_i\in\check{\P}^{10}$
of dimension~7 and degree~10 by a map that is generically 2:1. Hence the legs of a multipod with involutive displacements come in pairs:
if $(a,b,d)$ is a leg, then $(b,a,d)$ is also a leg. This can also be shown directly: if $\sigma\in\SE_3$ has order~2, then
\[ ||\sigma(a)-b|| = ||\sigma^2(a)-\sigma(b)|| = ||\sigma(b)-a|| . \]
Group-leg duality induces a duality between the projective subspace $L_i$ of dimension~10 that contains $X_i$ 
and the projective image space $\check{\P}^{10}$ that contains $Y_i$. Let us call the elements in $Y_i$ {\em twin pairs of legs};
each such pair of legs is constituted by a leg $(a,b,d)$ and by its conjugated leg $(b,a,d)$. Generically, three twin pairs in $Y_p$
correspond to three hypersurfaces in $L_i$. Since $\dim(X_i)=4$, the intersection of these three hypersurfaces and $X_p$ is a curve.
So, we have explained again the paradoxical mobility. 

But there is more. We have not just constructed a moving hexapod, we have even constructed, at the same time, a moving icosapod!
Here is the precise statement.

\begin{thm} \label{thm:bls}
Let $p_1,p_2,p_3$ be three generic twin pairs of legs. Let $C\subset X_s$ be the configuration curve of the hexapod defined
by all six legs. 
Then there exist seven additional twin pairs, maybe complex, such that $C$ is the set of
all order~2 displacements compatible with all 20 legs.
\end{thm}

\begin{proof}
The three twin pairs span a generic 2-plane in $V\subset\check{\P}^{10}$. The subvariety $Y_i\subset\check{\P}^{10}$ has dimension~7,
hence $V$ and $Y$ intersect in $\deg(Y_i)=10$ points. Three of them correspond to $p_1,p_2,p_3$, and the remaining seven are
the additional pairs we require. The linear span of all 10 points is equal to the linear span of $p_1,p_2,p_3$, namely $V$,
hence the conditions for displacements do not change.
\end{proof}

In \cite{Schicho:16d}, it is shown that there exist examples where all twenty legs are real. The proof is based on a result
on quartic spectahedra in \cite{DegtyarevItenberg:11,Ottem:15}.

\section{Compactification}
\label{sec:bond}

In enumerative algebraic geometry, for instance for the problem of counting rational curves on a projective variety,
compactifications of moduli spaces are known as a powerful tool. Here, we compactify the algebraic varieties
in which the configuration spaces are naturally embedded: products of subgroups of $\SE_3$ in the case
of linkages with revolute joints, $\SE_3$ itself in the case of multipods, and products of the plane 
in the case of moving graphs. 

\subsection{Moving Graphs}

The main theorem in  \cite{Schicho:19a} states that for a given graph, the existence of a flexible labeling is equivalent
to the existence of a NAC coloring. The construction of a flexible linkage from a given NAC-coloring was already 
explained in Section~\ref{sec:over}. For a construction proving the other implication, we need a compactification.

Let $(V,E,\lambda)$ be a graph with an edge assigment. We would like to projectivize in order to compactify; for this purpose,
it is convenient to slightly change the notion of a configuration slightly. A {\em homogeneous configuration} is an assignment of
vertices by points in $\R^2$ such that for any two edges $e=(i,j)$, $f=(k,l)$, the equality
\[ \lambda_e ||p_k-p_l||^2 = \lambda_f ||p_i-p_j||^2 \]
holds. For each vertex $i\in V$ with assigned point $p_k$, we write $p_k=(x_k,y_k)$ and $z_k:=x_k+\ci y_k$, $w_k:=x_k-\ci y_k$.
In other words, the complex numbers $z_1,\dots,z_{|V|}$ represent the vertices in the Gaussian plane of complex numbers.
In order to normalize, we require $p_1=(0,0)$.

The homogeneous configuration $p$ defines a point in $\P^{|V|-2}\times\P^{|V|-2}$ as follows: its first component has
projective coordinates $(z_2:\dots:z_{|V|})$, and its second component has coordinates $(w_2:\dots:w_{|V|})$. The equality above reads
\begin{equation} \label{eq:cf}
  \lambda_e (z_k-z_l)(w_k-w_l) - \lambda_f (z_i-z_j)(w_i-w_j) 
\end{equation}
in these projective coordinates. This is a bihomogeneous equation of bidegree $(1,1)$. The set of all solutions of (\ref{eq:cf})
is a projective subvariety of $\P^{|V|-2}\times\P^{|V|-2}$, the {\em configuration variety} of $(V,E,\lambda)$. Equivalent homogeneous
configurations define the same point in the configuration variety: since we fixed $p_1=(0,0)$, equivalent configurations are related
by a rotation or a scaling; but such a transformation just multiplies all $z$-coordinates by a complex nonzero constant and all
$w$-coordinates by a different complex nonzero constant, hence does not change the points in $\P^{|V|-2}$. 

A point $\alpha\in\P^{|V|-2}\times\P^{|V|-2}$ corresponds to a homogeneous configuration if and only if it fulfills two extra conditions.
First, the conjugate has to coincide with the flip of the first and second component; if this condition fails, then some of the corresponding
points in the plane have nonreal coordinates. Second, for some edge $e=(i,j)$, we have $(z_i-z_j)(w_i-w_j)\ne 0$. By (\ref{eq:cf}),
the choice of the edge has no influence on the validity of this extra condition. 

The boundary of the configuration set is defined as the set of points in the configuration variety that fail to satisfy the two
extra conditions. In particular, for some edge $e=(i,j)$, or equivalently for all edges, we have $(z_i-z_j)(w_i-w_j)=0$. For each point $\beta$
in the boundary, we define a coloring of the edges of the graph in the following way: the edge $(i,j)$ is colored red if $z_i-z_j$ 
vanishes at $\beta$, and blue otherwise.

\begin{lem} \label{lem:NAC}
For any point $\beta$ in the boundary of the configuration variety, the coloring defined by it is a NAC-coloring.
\end{lem}

\begin{proof}
Assume, indirectly, that all edges are red. Then the first projection of $\beta$ to $\P^{|V|-2}$ has only zero coordinates,
which is impossible.

Assume, indirectly, that all edges are blue. For any edge $(i,j)$, we have $(z_i-z_j)(w_i-w_j)=0$ and $z_i-z_j\ne 0$.
It follows that the second projection of $\beta$ to $\P^{|V|-2}$ has only zero coordinates, which is impossible.

Assume, indirectly, that $(i_1,\dots,i_k,i_1)$ is cycle such that $(i_r,i_{r+1})$ is red for all $r=1,\dots,k-1$, and $(i_k,i_1)$ is blue. 
Then $z_{i_1}=\dots=z_{i_k}$ and $z_{i_k}\ne z_{i_1}$, which is impossible.

Assume, indirectly, that $(i_1,\dots,i_k,i_1)$ is cycle such that $(i_r,i_{r+1})$ is blue for all $r=1,\dots,k-1$, and $(i_k,i_1)$ is red.
Then $w_{i_1}=\dots=w_{i_k}$, hence $w_{i_k}=w_{i_1}$. In addition, we also have $z_{i_k}=z_{i_1}$ as $(i_k,i_1)$ is red. Therefore the form
$(z_{i_1}-z_{i_k})(w_{i_1}-w_{i_k})$ vanishes with order $m\ge 2$ at $\beta$. The order of this form is the same for every edge, and because $(i_r,i_{r+1})$
is blue, the forms $z_r-z_{r+1}$ have order zero for $r=1,\dots,k-1$. Hence the order of the forms $w_r-w_{r+1}$ is at least $m$, for all $r$.
Then the form $w_{i_1}-w_{i_k}$ vanishes with order at least $m$, and this is a contradiction.
\end{proof}

\begin{thm} \label{thm:NAC}
A $(V,E)$ has a flexible labeling $\lambda$ if and only if it has a NAC-coloring.
\end{thm}

\begin{proof}
If $(V,E,\lambda)$ is flexible, then its configuration set is a projective variety $K$ of positive degree in $\P^{|V|-2}\times\P^{|V|-2}$. For any edge
$(i,j)\in E$, the form $(z_i-z_j)(w_i-w_j)$ has to vanish somewhere in $K$. Therefore, $K$ meets the boundary. By Lemma~\ref{lem:NAC}, it follows that
$(V,E)$ has a NAC-coloring.

Conversely, assume that we have a NAC-coloring of the edges. Then we  make the graph moving by a construction already given in Section~\ref{sec:gk}:
the red edges always keep their direction and move by translations only, while the blue edges rotate with uniform speed.
\end{proof}

For example, the graph in Figure~\ref{fig:nonac} has no NAC-coloring and therefore never moves for any labeling $\lambda$. 

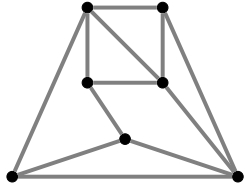
\begin{figure}[h]
\begin{center}
    \begin{tikzpicture}
        \node[vertex] (a) at (-0.5,-0.75) {};
                                \node[vertex] (b) at (0.5,0.5) {};
                                \node[vertex] (c) at (1.5,0.5) {};
                                \node[vertex] (d) at (2.5,-0.75) {};
                                \node[vertex] (e) at (0.5,1.5) {};
                                \node[vertex] (f) at (1.5,1.5) {};
                                \node[vertex] (g) at (1,-0.25) {};
                                \draw[edge]  (a)edge(d) (a)edge(e) (b)edge(c) (b)edge(e) (c)edge(d) (c)edge(e) (c)edge(f) (d)edge(f) (e)edge(f) (g)edge(a) (g)edge(b) (g)edge(d);
    \end{tikzpicture}
\end{center}
\caption{A graph that does not have a NAC coloring. Consequently, the graph is rigid for every possible labeling of its edges.}
\label{fig:nonac}
\end{figure}

A weakness of Theorem~\ref{thm:NAC} is that its constructive part -- the construction of flexible labelings -- produces only a particular type
of motions that we my call ``uniform speed motions''. Also, these motions sometimes map different non-adjacent vertices to the same point
in the plane. For example, in the case of the complete bipartite graph $K_{3,3}$, all uniform speed motions 
map at least two pairs of vertices to the same point in the plane, and the moving graph looks like a moving parallelogram. Deciding if a given
graph has labeling with a generically injective motion is much harder than deciding the existence of a flexible labeling; see \cite{Schicho:20b}.

\subsection{Revolute Loops}

The complete classification of mobile 4R loops was given by Delassus (see Section~\ref{sec:rl}).
The complete classification of mobile 5R loops was given in \cite{Karger:98} with the help of computer algebra.
For 6R loops, the classification is still open; the difficult part is to come up with necessary conditions for mobility.
In this subsection, we explain a method for deriving necessary criteria, which applies to $n$R loops for $n=4,5,6$.

We start with the closure equation expressed in algebraic way. Let $d_1,\dots,d_n$ (normal distances), $s_1,\dots,s_n$ (offsets), 
and $w_1,\dots,w_n$ (cotangents of half angles) the invariant Denavit-Hartenberg parameters. For $r=0,\dots,n-1$, the dual
quaternion $g_r := (1-s_r\eps\qi)(w_r-\qk)(1-d_i\qk)$ is the displacement that transforms the internal coordinate system of
link $r$ to the internal coordinate system of link $r+1$ (modulo $n$), if the configuration parameter is zero.
The closure equation is an equation in the variables
$t_0,\dots,t_{n-1}$, which denote the cotangents of the half configuration angles: the dual quaternion
\[ x(t_0,\dots,x_{n-1}) := (t_0-\qi )g_0(t_1-\qi )g_1\cdots (t_{n-1}-\qi )g_{n-1} \]
is a multiple of $1$, hence 7 of its 8 coefficients are zero. The variables $t_0,\dots,t_{n-1}$ may also assume the value $\infty$;
in this case, the corresponding factor $(t_r-\qi)$ is replaced by the scalar 1, or is simply omitted. In this section,
we will avoid this technicality.

We focus on solutions on the boundary, but this time we do not consider $t_r=\infty$ as boundary. Instead, we define the boundary
of $(\P^1)^n$ as the set of $n$-tuples $(t_0,\dots,t_{n-1})$ such that $t_r^2+1=0$ for at least one $r$. Indeed, if we
remove the boundary, then we get a group variety isomorphic to $(\SO_2)^n$, with an isomorphism respecting real structures.
The statement that $t_r^2+1=0$ for at least one $r$ is equivalent
to the statement $N(x(t_0,\dots,t_n))=0$, by the multiplicativity of the norm. Boundary solutions can never be real,
at least one of the variable must be equal to $\pm\ci$. 

{\bf Note:} Throughout this paper, we use $\qi$ for the first quaternion unit in $\H$, $\ci$ for the imaginary unit in $\C$,
and $i$ for a running integer. In this sections, both $\qi$ and $\ci$ will appear, sometimes in the same expression; but
we will try to avoid using $i$ for an integer.

Unfortunately, the closure equation often has many solutions that are not of interest. But we can obtain more equations
by cyclic permutation of its factors, or by using quaternion conjugation to bring some factors to the other side, as in
\[ \lambda (t_0-\qi)g_0= \nu \overline{g_{n-1}}(t_{n-1}+\qi)\overline{g_{n-2}}\cdots \overline{g_1}(t_1+\qi) \]
for some scalars $\lambda,\nu$ that are not both equal to zero. This condition can be expressed by polynomial equations,
namely the $2$-minors of the $2\times 8$ matrix whose rows are the coordinates of $(t_0-\qi)g_0$ and of
$\overline{g_{n-1}}(t_{n-1}+\qi)\overline{g_{n-2}}\cdots \overline{g_1}(t_1+\qi)$.
After having added all these reformulation of the closure equations to our system of
equations, we look for solutions on the boundary. These are called {\em bonds}. 

It is an easy exercise to prove that at least two of $t_0,\dots,t_{n-1}$ must be $\pm\ci$. {\em Hint:} use a formulation of the
closure equation with factors on both sides, and then take the norm on both sides. There are many examples with
exactly two of $t_0,\dots,t_{n-1}$ being $\pm\ci$. If, say, $t_1^2+1=t_k^2+1=0$ for some $k<n$, and $t_i^2+1\ne 0$ for $i\ne 1,k$,
then we say that the first joint and the $k$-th joint are {\em entangled} in the respective bond. 
We can then prove the following equations:
\begin{equation} \label{eq:bond}
\begin{aligned}
  (t_1-\qi )g_1(t_2-\qi)g_2\cdots (t_k-\qi ) = 0,  \\
  (t_k-\qi )g_k(t_{k+1}-\qi)g_{k+1}\cdots (t_0-\qi )g_n(t_1-\qi ) = 0. 
\end{aligned}
\end{equation}
If the number of coordinates $t_r$ with $t_r^2+1=0$ is bigger than two, then Equation~\ref{eq:bond} also holds form some
$k$, up to cyclic permutation by \cite[Lemma~2 and Theorem~3]{hss2}. 

Equation~\ref{eq:bond} together with $t_1^2+1=t_k^2+1=0$ is quite restrictive and often has implications on the invariant
parameters that are hard-coded in $g_0,\dots,g_{n-1}$. The case $k=2$ is easy to analyze: assume 
\[ (\ci-\qi)g_1(\ci-\qi) = 0 . \]
Then it follows that $w_1=d_1=0$; geometrically this means that the first two rotation axes are equal, except that they have
opposite orientation in the closure equation. We may exclude this degenerate case, and then we always have $k>2$
(and modulo $n$, this also excludes $k=0$).

If $k=3$, then we get the equation
\[ (\ci-\qi)g_1(t_2-\qi)g_2(\ci-\qi) = 0, \]
up to orientation of the first and/or third axis. This is a system of inhomogeneous linear equations for $t_2$.
It has a solution in three cases: either the three axes are parallel, or the three axes are concurrent, or the equations
\[ s_2=0, \frac{d_1}{\sin(\alpha_1)}=\frac{d_2}{\sin(\alpha_2)} \]
are true. This should be compared with Bennett's condition for a 4R loop to be mobile in Section~\ref{sec:rl}: it is
exactly the condition that has to be fulfilled for three axis that is true if and only if there exists a 4th line
that supplements the three lines to a mobile 4R loop. The ``Benett condition for three lines'' mysteriously appears
in Dietmaier's collection \cite{Dietmaier:95} of known families of 5R loops and 6R loops. The bond equation explains why: in a mobile
linkage, bonds have to be present, and for each bond there must be two non-neighboring joints entangled in a bond.
In a 5R loop, we then have $k=3$ up to a cyclic permutation. In a 6R loop, we have either $k=3$ -- entanglement
of diagonal joints --, or $k=4$, entanglement of opposite joints.
Many known families have some bond that entangles diagonal joints.

The analysis of the case $k=4$ is more involved; however, it is necessary in order to explain mobility of 6R linkages
that have no three consecutive axis fulfilling the Bennett condition for three lines. Assume that $n=6$, and we have
a bond $\vec{t}$ that entangles the first and the fourth joint. Without loss of generality, we may assume $t_1=t_4=\ci$. 
Then we obtain the equations
\begin{equation}
\begin{aligned}
  (\ci-\qi)g_1(t_2-\qi)g_2(t_3-\qi)g_3(\ci-\qi) = 0,  \\
  (\ci-\qi)g_4(t_5-\qi)g_5(t_0-\qi)g_0(\ci-\qi) = 0.  
\end{aligned}
\end{equation}
Excluding some degenerate cases (4 parallel lines, or 4 lines meeting in a point), the first equation allows two solutions
for $(t_2,t_3)$, while the second equation allows two solutions for $(t_0,t_5)$. These partial solutions are not independent.
They have to satisfy another reformulation of the closure equation:
\[ \lambda(\ci-\qi)g_1(t_2-\qi)g_2(t_3-\qi)g_3 = \nu\overline{g_0}(t_0+\qi)\overline{g_5}(t_5+\qi)\overline{g_4}(\ci+\qi) , \]
for some complex numbers $\lambda,\nu$ that are not both equal to zero. By resultants, we can eliminate the variables $t_0,t_2,t_3,t_5$
and obtain an equivalent formulation without these variables: the two quadratic polynomials
\[ Q_1^+(x) = \left(x+\frac{b_{3}c_3-b_{1}c_1}{2}-\frac{s_{1}}{2}\ci\right)^2 + \]
 \[ \frac{\ci}{2}\left(b_1 s_{2}+b_{3} s_{3}+s_{2} b_{3} c_{2}+s_{3} b_1 c_{2}\right) -\]
 \[ \frac{b_1 b_{3} c_{2}-s_{2} s_{3} c_{2}}{2}
   + \frac{s_{2}^2+s_{3}^2-b_1^2+b_{2}^2-b_{3}^2-b_{2}^2 c_{2}^2}{4}, \]
\[ Q_4^+(x) = \left(x+\frac{b_0c_0-b_{4}c_4}{2}-\frac{s_{4}}{2}\ci\right)^2 + \]
 \[ \frac{\ci}{2}\left(b_4 s_{5}+b_0 s_0+s_{5} b_0 c_{5}+s_0 b_4 c_{5}\right) -\]
 \[ \frac{b_4 b_0 c_{5}-s_{5} s_0 c_{5}}{2}
   + \frac{s_{5}^2+s_0^2-b_4^2+b_{5}^2-b_0^2-b_{5}^2 c_{5}^2}{4} \]
in $\C[x]$ have a common zero; here, $c_i:=\cos(\alpha_i)$ and $b_i:=\frac{d_i}{\sin(\alpha_i)}$ for $i=1,\dots,6$.
For the details this elimination of variables, we refer to \cite{Schicho:15e}.

If one of the coordinates of $t_1$ and $t_4$, or both, are equal to $-\ci$, then there are two quadratic polynomials that
are similarily defined, having a common zero. In total, the number of bonds entangling the first joint and fourth joint is even 
(because bonds always appear in complex conjugate pairs) and at most 8. It is equal to 8 if and only if the two 
polynomials are equal in each of the four pairs of quadratic polynomials. 

Suppose that we have the maximal number of 8 bonds entangling opposite axes, for all three pairs of opposite axes.
This assumption leads to a system of algebraic equations in the
invariant Denavit/Hartenberg parameters (18 variables). Using computer algebra, we can compute the solution set
(see \cite{Schicho:15e}). It turns out that there are two components $F_1$ and $F_2$, of dimension 6 and 7, respectively. 
Both are families of mobile 6R loops that have not been known before bonds were used in kinematics. But the family $F_1$ (the one
of dimension 6) has a 5-dimensional subfamily which is classical: Bricard's orthogonal 6R loops, characterized by
the vanishing of $c_0,\dots,c_5$  (i.e., all angles are right angles) and $s_0,\dots,s_5$ (i.e., all offsets are zero), 
and the single equation $b_0^2-b_1^2+b_2^2-b_3^2+b_4^2-b_5^2=0$.

\subsection{Multipods}

The two varieties that play a role in the analysis of multipods, namely the group variety $X\in\P^{16}$ and the 
leg variety $Y\in\check{P}^{16}$, both come with a natural definition of a boundary: the boundary of $X$ is defined
by the linear equation $h=0$ and the boundary of $Y$ is defined by the linear equation $u=0$, with the variable
setting as in Section~\ref{sec:pod}. The group variety is more interesting, because the configuration set of a mobile multipod 
-- defined as the intersection of $X$ with hyperplanes dual to the legs --
will always intersect the boundary $h=0$. The leg set of a mobile multipod, on the other hand, might be disjoint
from the boundary $u=0$.

Let us have a closer look at the boundary $B:=X\cap H$, where $H$ is the hyperplane $H:h=0$. We refer to \cite{Schicho:14c}
for the calculation; here we report on only the facts we will use later. First, $B$ is a variety of dimension~5 and degree~20.
The variety $X$ -- which has degree~20 -- and the hyperplane $H$ intersect tangentially along $B$, with intersection multiplicity~2.
The boundary $B$ has a natural decomposition into five locally closed subsets, which we denote by $Z_i$, $Z_b$, $Z_s$, $Z_c$, and $Z_v$.
The stratum $Z_i$ has dimension~5 and consists of all points in $B$ which are smooth points of $X$ such that at least one of the 
$m_{ij}$-coordinates is not zero. The stratum $Z_i$ has dimension~4 and consists of all points in $B$ which are singular points of $X$
such that at least one of the $m_{ij}$-coordinates is not zero. The stratum $Z_s$ has dimension~3 and consists of all remaining boundary
points such that one of the coordinates $x_1,x_2,x_3$ is not zero and one of the coordinates $y_1,y_2,y_3$ is not zero. The stratum
$Z_c$ has dimension~2, and here one of the previous three triples of coordinates has values all zero. Finally, the stratum $Z_v$
consists of a single point: all coordinates except $r$ are zero. It is the only point on $B$ defined over the reals, all other
boundary points are complex.

For a multipod given by a set of legs, we have a set of hyperplanes in $\P^{16}$ dual to the legs. We now define the set of bonds 
of the multipod as the set of intersection points of all these hyperplanes with the boundary $B$. 
The main point of the analysis of the boundary is that the presence of bonds in particular strata has geometric implications
for the geometry of the legs. Let us show this for the stratum $Z_s$.
Here, the projections  $x:=(x_1:x_2:x_3)$ and $y:=(y_1:y_2:y_3)$ both exist, because there is at least 
one in both triples of coordinates that does not vanish. From the equations, it is easy to derive that both $x$ and $y$
have to lie on the absolute conic $x_1^2+x_2^2+x_3^2=0$, which clearly has no real points. 

\begin{thm} \label{thm:simbond}
Let $\{ (a_l,b_l,d_l) \mid l\in L\}$ be the leg set of a multipod, where $L$ is an index set parametrizing the legs.
Assume that this multipod has a bond in $Z_s$.
Then there exist orthogonal projections $p_a:\R^3\to\R^2$ and $p_b:\R^3\to\R^2$
and a similarity transformation $s:\R^2\to\R^2$ such that $s(p(a_l))=p(b_l)$ for all $l\in L$.
\end{thm}

\begin{proof}
The variety $X$ has an automorphism group of dimension~12, by left and right multiplication of group elements.
The statement of the theorem is invariant under these automorphism. We use suitable automorphisms to transform the bond in $Z_s$ to 
a point with coordinates $(x_1,x_2,x_3)=(1,\ci,0)$, $(y_1,y_2,y_3)=(\lambda,\lambda\ci,0)$ and all remaining coordinates zero; it is
important that this transformation can be achieved by real transformations (see \cite{Schicho:14c} for the calculation).
The corresponding hyperplane in leg space has equation
\[ a_1+a_2\ci +\lambda b_1+\lambda b_2\ci = 0 . \]
For all $l\in L$, the leg $(a_l,b_l,d_l)$ must lie on this hyperplane. The real part and the imaginary part of this equation must
both be zero: $a_{l,1}+\lambda b_{l,1}=a_{l,2}+\lambda b_{l,2}=0$. Therefore the claim follows.
\end{proof}

The stratum $Z_s\subset B$ is also called the {\em similarity stratum}, and the any bond in this stratum
is called a {\em similarity bond}. The theorem above states, in more informal language, that the presence of a similarity
bond implies the existence of two similar projections of base and platform. 
If a linkage has an infinite number of similarity bonds, then it can be
shown that for all projections of the platform points, there is a similar projection of the base points. This implies that 
the platform points and the base points are related by a similarity transformation of $\R^3$. This geometric observation plays a crucial role
in the classification of multipods of mobility~2 in \cite{Schicho:15c}.

There is the analoguous statement for the stratum $Z_i$; also the proof is analoguous.

\begin{thm} \label{thm:invbond}
Assume that the multipod above has a bond in $Z_i$.
Then there exist orthogonal projections $p_a:\R^3\to\R^2$ and $p_b:\R^3\to\R^2$
and an inversion $i:\R^2\to\R^2$ such that $i(p(a_l))=p(b_l)$ for all $l\in L$.
\end{thm}

Recall the Bricard-Borel multipod with infinitely many legs, described i Section~\ref{sec:pod}: all its legs $(a,b,d)$ satisfy the condition
\[ a_1b_2-a_2b_1 = a_1b_1+a_2b_2-\alpha = 0 \]
for some fixed $\alpha\in\R$, $\alpha\ne 0$. As we already saw, this is an inversion relation between the projections of base and platform.

If a multipod has a bond in $Z_b$, then there are two lines $G_a,G_b\subset\R^3$ such that for any leg $(a,b,d)$ in the dual hyperplane in leg space,
either $a$ lies in $G_a$ or $b$ lies in $G_b$. The presence of a bond in $Z_b$ implies the existence of two lines and a partition
of the set of legs into two subsets, with the first subset having collinear anchor points in the base and the second subset
having collinear anchor points in the platform. Let us called such a configuration a {\em combined collineation}. 
The existence of such a combined collineation already implies mobility for a suitable choice of leg lengths. 
To see this, we start with a configuration such that the lines $G_a$ and $G_b$ coincide -- the leg lengths have to
chosen so that such a configuration exists. Then we can rotate the platform around this line
(similar as the ``double banana'' Figure~\ref{fig:l3d}). 
%In Karger, this motion is called a {\em butterfly motion}, and the name
%was extended to the bonds of such motion.

The stratum $Z_c$ has two irreducible components. For one of these components, the projection $(x_1:x_2:x_3)$ is defined. 
The geometric implication is stronger than the implication from a bond in $Z_b$: all anchor points in the base have to be collinear.
For the second component, all anchor points in the platform have to be collinear.
If one of these two conditions is fulfilled, then a rotation motion is possible from any starting position. 

The hyperplane corresponding to the one point in $Z_v$ is the hyperplane at infinity. Hence there is no multipod with a finite leg
that has a bond in $Z_v$.

In summary, any boundary point implies {\em some} geometric condition on the two configurations of platform points 
and of base points. Hence the compactification gives necessary conditions for mobility: if a multipod is mobile,
then it must have some bonds, therefore one of the above geometric conditions hold. 

Many mobile multipods
have more than just one pair of complex conjugate bonds, since the number of bonds is related to the degree
of the configuration curve embedded in $\P^{16}$.
The correlation between the degree of the mobility curve of a hexapod and the number of special geometric events --
similar projections, inverse projections, or combined collineations -- motives the question on the maximal number of such events.
Here are the answers.

\begin{thm}
Assume that the six-tuple of points in the base and the six-tuple of points in the platform are not similar, and that
neither the base nor the platform consist of coplanar points.

a) The number of combined collineations is at most 16. If every anchor point appears in at most one leg, for both base and platform,
   then the maximal number of combined collineations is 4.

b) The number of projections related by a similarity transformations is at most 2. The maximum is reached if and only if the two six-tuples are
   affine equivalent.

c) The number of projections related by an inversion is at most 7.
\end{thm}

The proofs of (a) and (b) are left as exercises. For the proof of (c), we refer to \cite{Schicho:17b}.

It is conjectured in \cite{Schicho:17b} that for a generic choice of six points in $\R^3$, there exists a second six-tuple of points, such
that the maximal number of 7 projections related by an inversion is reached; such a six-tuple would then be unique up to similarity. 
The conjecture continues to state that there is a unique scaling and choice of leg length such that the so constructed hexapod is mobile, with a mobility
curve of maximal degree~28. 
For a numeric random choice, the conjecture can be tested by a construction taking about 300 seconds using computer algebra.
Using this construction, the conjecture has been verified for 50 random choices. Theoretically, it is still theoretically possible
that these 50 random choices were picked on some unknown subvariety with non-generic behavior, but it is quite unlikely.

\end{document}